\documentclass[11pt,a4paper]{article}
\usepackage{graphicx}
\usepackage{soul}
\usepackage{multirow}
\usepackage{amssymb}
\usepackage[colorlinks=true,linkcolor=blue,citecolor=red]{hyperref}
\usepackage{tikz}
\usetikzlibrary{calc,trees,positioning,arrows,chains,shapes.geometric,%
    decorations.pathreplacing,decorations.pathmorphing,shapes,%
    matrix,shapes.symbols}
\tikzset{
>=stealth',
  punktchain/.style={
    rectangle, 
    rounded corners, 
     fill=cyan!40,
    draw=black, very thick,
    text width=10em, 
    minimum height=2em, 
    text centered, 
    on chain},
  line/.style={draw, thick, <-},
  element/.style={
    tape,
    top color=white,
    bottom color=blue!50!black!60!,
    minimum width=8em,
    draw=blue!40!black!90, very thick,
    text width=10em, 
    minimum height=2.5em, 
    text centered, 
    on chain},
  every join/.style={->, thick,shorten >=1pt},
  decoration={brace},
  tuborg/.style={decorate},
  tubnode/.style={midway, right=2pt},
}
\usepackage{standalone}
\usepackage{epstopdf}
\usepackage[english]{babel}
\usepackage{latexsym}
\usepackage{subfigure}
\usepackage{color}
\usepackage{amsmath,amsfonts,amsthm}

\newtheorem{rem}{Remark}

\newtheorem{prop}{Proposition}
\newtheorem{thm}{Theorem}

\newtheorem{alg}{\small{Algorithm}}

\def\be{\begin{equation}}
\def\ee{\end{equation}}
\def\bea{\begin{eqnarray}}
\def\eea{\end{eqnarray}}

\def\RR{\mathbb R}

\begin{document}

\title{Particle based gPC methods for mean-field models of swarming with uncertainty}
\date{}

\author{Jos\'e Antonio Carrillo \thanks{Department of Mathematics, Imperial College London, SW7 2AZ, UK ({\tt carrillo@imperial.ac.uk})} \and Lorenzo Pareschi \thanks{Department of Mathematics and Computer Science, University of Ferrara, Via Machiavelli 35, 44121 Ferrara, Italy ({\tt lorenzo.pareschi@unife.it}).} \and
Mattia Zanella\thanks{Department of Mathematical Sciences, Politecnico di Torino, Corso Duca degli Abruzzi 24, Torino, Italy ({\tt mattia.zanella@polito.it}).}}

\maketitle
\begin{abstract}
In this work we focus on the construction of numerical schemes for the approximation of stochastic mean-field equations which preserve the nonnegativity of the solution. The method here developed makes use of a mean-field Monte Carlo method in the physical variables combined with a generalized Polynomial Chaos (gPC) expansion in the random space. In contrast to a direct application of stochastic-Galerkin methods, which are highly accurate but lead to the loss of positivity, the proposed schemes are capable to achieve high accuracy in the random space without loosing nonnegativity of the solution. Several applications of the schemes to mean-field models of collective behavior are reported. \\


\end{abstract}

\section{Introduction}
\label{sec1}
An increasing number real world phenomena have been fruitfully described by kinetic and mean-field models. Particular attention has been paid in the past decade to self-organizing systems in socio-economic and life sciences. Without intending to review the very huge literature on these topics, we refer the reader to \cite{APTZ,BCCD,CFTV,CPT,CristPiccTos02,EHS,MT,PZ1,T} and the references therein. 

Kinetic models may be derived in a rigorous way from microscopic particle dynamics in the limit of a large number of agents \cite{BolCa,CFTV,HT,JW,PT,T}. It is a well known fact that the main disadvantage of the microscopic approach to capture the asymptotic behavior of interacting systems relies on the so-called curse of dimensionality. For example, if we consider $N$ interacting individuals the cost is of order $O(N^2)$, becoming rapidly unaffordable in the case of large systems. For this reason, kinetic and mean-field type modeling have been developed to represent the evolution of distribution functions obtained in the asymptotic regimes, which of course become independent of the size of the system.

The introduction of uncertainty in the mathematical modeling of real world phenomena seems to be unavoidable for applications. In fact we can often have at most statistical information
of the modeling parameters, which must be estimated from experiments or derived from heuristic observations \cite{Ball_etal,BFHM,KTIHC}. Therefore, to produce effective predictions and to better understand physical phenomena, we can incorporate all the ineradicable uncertainty in the dynamics from the beginning of the modeling.

In the following a formal derivation of uncertain mean-field equations for a class of microscopic models for alignment is proposed.  At the numerical level one of the most popular techniques for uncertainty quantification is based on stochastic Galerkin (SG) methods. In particular, generalized polynomial chaos (gPC) gained increasing popularity in UQ, for which spectral convergence on the random field is observed under suitable regularity assumptions \cite{DPZ,HJ,HJX,X,XK,LeMK,ZJ}. Nevertheless, these methods need a strong modification of the original problem and when applied to hyperbolic and kinetic problems lead to the loss of some structural properties, like positivity of the solution, entropy dissipation or hyperbolicity, see \cite{DPL}. 

Beside SG-based methods, non-intrusive approaches for UQ have been developed in recent years like the stochastic collocation (SC) methods \cite{DPZ,NZ,X,XH}. These methods have the nice feature to keep the structural properties of the underlying numerical solver for the deterministic problem.

In this work we focus on the construction of numerical schemes which preserve the positivity of relevant statistical quantities, keeps the high accuracy typical of gPC approximations in the random space and takes advantage of the reduction of computational complexity of Monte Carlo (MC) techniques in the physical space \cite{AlbiPareschi13,Caflisch,DP15,PT}. We investigate these Monte Carlo gPC (MCgPC) methods for mean-field type equations, which permits with a cost of $O(SN)$, $S\ll N$,  to obtain a positive numerical approximation of expected quantities. 

The rest of the paper is organized as follows. In Section 2 we introduce the microscopic models of swarming with random inputs and review some of their main properties and their mean-field limit. Section 3 is devoted to the construction of numerical methods for uncertainty quantification. We first survey some results on gPC expansions and derive the classical stochastic Galerkin scheme for the mean-field problem. Subsequently we describe the new particle based gPC approach. Finally, in Section 4 several numerical results are presented which show the efficiency and accuracy of the new Monte Carlo-gPC approach.  


\section{Microscopic and mean-field models with uncertainty}
\label{sec2}
In the following we introduce some classical microscopic models of collective behavior \cite{CFTV,CS,DOCBC} in the stochastic case characterized by random inputs. In collective motion of groups of animals three zones are distinguished: the first is the repulsion region, where agents try to avoid physical collisions, hence in the immediate proximity they align to the possible direction of the group and, at last, the attraction region, where individuals too far from the group are attracted to the collective center of mass.

The microscopic description of a dynamical systems composed of $N$ individuals is described by a system of ordinary differential equations for $(x_i(\theta,t),v_i(\theta,t))\in \RR^{d_x}\times \RR^{d_v}$, $i=1,\dots,N$ with the general structure
\begin{equation}\begin{cases}\label{eq:system}
\dot{x}_i(\theta,t)=&v_i(\theta,t),\\
\dot{v}_i(\theta,t)=&S(\theta;v_i)+\dfrac{1}{N}\displaystyle\sum_{j=1}^N \Big[H(\theta;x_i,x_j) (v_j(\theta,t)-v_i(\theta,t))\\
&\qquad\qquad+A(\theta;x_i,x_j)+R(\theta;x_i,x_j)\Big],
\end{cases}
\end{equation}
where $x_i(\theta,0)=x_i^0$, $v_i(\theta,0)=v_i^0$ denote the initial positions and velocities of the agents and we introduced the functions depending on the random input $\theta$: $H(\theta;x_i,x_j)$,  representing the alignment process, $A(\theta;x_i,x_j)$ the attraction dynamics and the term $R(\theta;x_i,x_j)$ the short-range repulsion.  In \eqref{eq:system} the function $S(\theta;v_i)$ describes a self-propelling term.  

We will assume through the paper that the stochastic quantity $\theta$ is distributed according to a known probability density $g:\RR\rightarrow\RR^+$, such that $g(\theta)\ge 0$ a.e. in $\RR$, $\textrm{supp}(g)\subseteq I_{\Theta}$ and $\int_{I_{\Theta}}g(\theta)d\theta=1$.

\subsection{Cucker-Smale dynamics with uncertainty}

In the case of flocking dynamics \cite{AH,CFRT,CS} the interaction function depends on the Euclidean distance between two agents, i.e. $H(\theta;x_i,x_j)=H(\theta;|x_i-x_j|)$ is of the form 
\be\label{eq:HCS}
H(\theta;|x_i-x_j|) = \dfrac{K}{(1+|x_i(\theta,t)-x_j(\theta,t)|^2)^{\gamma}},
\ee
with $K,\gamma>0$ which may depend on stochastic inputs. We will refer to a system of agents with trajectories \eqref{eq:system} and interaction function of the form \eqref{eq:HCS} as Cucker-Smale (CS) model. Further, we set in \eqref{eq:system} $A(\cdot;\cdot,\cdot)\equiv 0$ and $R(\cdot;\cdot,\cdot)\equiv 0$.

In the deterministic setting, different regimes are described by the introduced model in relation to the choice of $K,\gamma>0$. 
In particular, the following result has been proven in \cite{CS}, see Theorem 2 p. 855. We also refer the reader to \cite{CFRT,HL,HT} for related results and improvements. 

\begin{thm}
Assume that one of the following conditions holds
\begin{itemize}
\item[i)] $\gamma\leq 1/2$
\item[ii)] $\gamma> 1/2$ and
\be\label{eq:condition_th}
\left[ \left(\dfrac{1}{2\gamma}\right)^{\frac{1}{2\gamma-1}}-\left(\dfrac{1}{2\gamma}\right)^{\frac{2\gamma}{2\gamma-1}} \right] \left(\dfrac{K^2}{8N^2\Lambda(v_0)}\right)^{\frac{1}{2\gamma-1}}> 2\Gamma(x_0)+1,
\ee
where 
$$
\Gamma(x) = \dfrac{1}{2}\sum_{i\ne j}|x_i-x_j|^2,\qquad \Lambda(v)=\dfrac{1}{2} \sum_{i\ne j}|v_i-v_j|^2.
$$
\end{itemize}
Then there exists a constant $B_0$ such that $\Gamma(x(t))\le B_0$ for all $t\in\RR^+$ while $\Lambda(v(t))$ converges toward zero as $t\rightarrow +\infty$, and the vectors $x_i-x_j$ tend to a limit vector $\hat{x}_{ij}$, for all $i,j\le N$.
\end{thm}

Therefore in the case $\gamma\le 1/2$ we will refer to unconditional alignment (or flocking) given that the velocities alignment does not depend on initial configuration of the system or on dimensionality. In this case all the agents of the population have the same velocity, they form a group with fixed mutual distances with a spatial profile which depends on the initial condition. If $\gamma>1/2$ flocking may be expected under the condition of equation \eqref{eq:condition_th}. 

 \begin{prop}\label{prop1}
Let us consider the evolution of the stochastic CS model \eqref{eq:system} with interaction function of the form \eqref{eq:HCS} and $K=K(\theta)$, i.e.  for a deterministic $\gamma\le 1/2$
$$
\begin{cases}
 \dot { x}_i(\theta,t) =  v_i(\theta,t), \qquad i = 1,\dots,N\\
 \dot { v}_i(\theta,t) = \dfrac{1}{N}\displaystyle\sum_{j=1}^N \dfrac{K(\theta)}{(1+|x_i(\theta,t)-x_j(\theta,t)|^2)^{\gamma}} (v_j(\theta,t)-v_i(\theta,t)),
 \end{cases}
$$
 subject to deterministic initial conditions $x_i(\theta,0)=x_i^0$, $v_i(\theta,0)=v_i^0$ for all $i=1,\dots,N$. 
 The support of the velocities exponentially collapse for large time for each $\theta\in I_{\Theta}$ provided $K(\theta)>0$. 
\end{prop}
\begin{proof}
We omit the proof that is reminiscent of well established results, see for example Proposition 5 p. 231 of \cite{CFRT}. Similar results have been also obtained in \cite{HL,MT}. 
\end{proof}

In the case $K=K(\theta)$ and $\gamma=\gamma(\theta)>0$ we can study the behavior of the system in a neighbor of the deterministic value $\gamma_0\le 1/2$ for which unconditional alignment does emerge provided $K>0$.  We can prove the following result.
\begin{prop}
In a neighbor of $\gamma_0\le 1/2$ a linearization of the uncertain CS model with $K(\theta)>0$, $\gamma(\theta)>0$ reads
\be\label{eq:CS_lin}
\begin{cases}
\dot x_i(\theta,t) &= v_i, \qquad i=1,\dots,N \\
\dot v_i(\theta,t) &= \dfrac{1}{N}\displaystyle\sum_{j=1}^N \dfrac{K(\theta)}{(1+|x_i-x_j|^2)^{\gamma_0}}\left(1-(\gamma(\theta)-\gamma_0)\log(1+|x_i-x_j|^2)\right)(v_j-v_i),
\end{cases}
\ee
for which unconditional flocking of the velocities follows if 
\be\label{eq:gamma_lin}
\gamma(\theta)< \gamma_0.
\ee
\end{prop}
\begin{proof}
We can linearize the interaction function $H(\theta;\,\cdot)$ in a neighbor of $\gamma_0\le 1/2$ as follows
\[
H(\theta;|x_i-x_j|) = \dfrac{K(\theta)}{(1+|x_i-x_j|^2)^{\gamma_0}}+\dfrac{\partial \bar H}{\partial \gamma}(\theta;|x_i-x_j|)(\gamma(\theta)-\gamma_0),
\]
being $\bar H = K(\theta)/(1+|x_i-x_j|^2)^{\bar \gamma}$ and where $\bar \gamma = \lambda\gamma_0+(1-\lambda)\gamma(\theta)$, $\lambda\in[0,1]$. Hence, it is seen that the linearized system assumes the form \eqref{eq:CS_lin} for which we impose the positivity of the strength of interaction, condition that leads to \eqref{eq:gamma_lin}. 
\end{proof}

\subsection{Other swarming models with uncertainty}\label{sec:other_mod}
The microscopic model introduced by D'Orsogna, Bertozzi et al. in \cite{DOCBC} describes dynamics which self-propulsion, attraction and repulsion zones. The model is given as follows
\be\label{eq:DOB}
\begin{cases}
\dot x_i(\theta,t) = v_i(\theta,t),\\
\dot v_i(\theta,t) = (a-b|v_i(\theta,t)|^2)v_i(\theta,t) - \dfrac{1}{N} \displaystyle\sum_{j\ne i}\nabla_{x_i}U(\theta;|x_j(\theta,t)-x_i(\theta,t)|)
\end{cases}
\ee
for all $i=1,\dots,N$. In the system of differential equations \eqref{eq:DOB} the quantities $a$, $b$ are nonnegative parameters representing the self-propulsion of individuals and a friction term following Rayligh's law respectively. Further, $U:\RR^{2d}\times I_{\Theta} \rightarrow \RR$ is an uncertain potential modeling short-range range repulsion and long-range attraction. A typical choice for the potential $U$ is a Morse potential of the form 
\be \label{eq:U_Morse}
U(\theta;r) = -C_A(\theta)e^{-r/\ell_A} + C_R(\theta)e^{-r/\ell_R},
\ee
where $C_A(\theta)$, $C_R(\theta)$, $\ell_A$, $\ell_B$ are the uncertain strength and length of attraction/repulsion respectively. In collective behavior description of interest is the case $C(\theta):=C_R/C_A>1$ and $\ell:=\ell_R/\ell_A<1$ corresponding to long-range attraction and short-range repulsion. It is well-known that several equilibria may be described through this system: the first of stability for $C(\theta)\ell^{2d}>1$ for all $\theta\in I_{\Theta}$, with agents forming a crystalline pattern, whereas if $C(\theta)\ell^{2d}<1$ the agents tend to a single or double mills of constant speed $v=\sqrt{a/b}$, see \cite{CFTV}. We may similarly consider the case where uncertainties are present also in the self-propelling term or in the characteristic lengths of attraction/repulsion as well, anyway we will limit to the one described here.

\begin{rem}
As a modification of the classical CS model recent literature considered the case of non-symmetric interactions \cite{MT}. The authors considered the case where alignment is based on a relative influence between agents, its version with uncertainty reads as follows for all $i=1,\dots,N$
$$
\begin{cases}
\dot x_i(\theta,t) = v_i(\theta,t), \\
\dot v_i(\theta,t) = \dfrac{1}{N}\displaystyle\sum_{j=1}^N h(\theta;x_i,x_j)(v_j(\theta,t)-v_i(\theta,t)),
\end{cases}
$$
where $h(\theta;\cdot,\cdot)$ express an uncertain relative alignment strength between the agents $i,j$ and is defined as
$$
h(\theta;x_i,x_j) = \dfrac{H(\theta;|x_i-x_j|)}{\tilde{H}(\theta;x_i)},\qquad \tilde{H}(\theta;x_i) = \dfrac{1}{N} \sum_{k=1}^N H(\theta;|x_i-x_k|),
$$
and $H(\theta;\cdot)$ is given by \eqref{eq:HCS}. The introduced model prescribes weighted interactions between the agents of the system, this symmetry breaking of the original CS dynamics links this problem to more sophisticated modeling, for example leader-follower models and limited perception models as well. We refer the reader to \cite{CFTV,Shen08} for more details and further references. 
\end{rem}
\subsection{Mean-field limit}\label{sec:MF}
In the case of very large number of interacting individuals the numerical solution of the coupled system of ODEs poses serious problems due to the curse of dimensionality. For this reason the description of the interacting systems at different scales \cite{AlbiPareschi13,CFTV,HT} is of primary importance. Therefore, we tackle the problem by considering the distribution function of particles dependent on the stochastic variable $\theta\in I_{\Theta}$ $f(\theta,x,v,t)\ge 0$ with position $x\in\RR^{d_x}$, $v\in\RR^{d_x}$ at time $t\ge 0$. The evolution of $f$ is then derived from the microscopic dynamics via asymptotic techniques.

To obtain the mean-field formulation of the CS dynamics with stochastic interactions we can follow the usual approaches for all $\theta\in I_{\Theta}$, see  \cite{CFTV,CHL,HT}. Let us consider the system of ODEs \eqref{eq:system}, a possible way to derive the corresponding evolution for $f\ge 0$ relies in BBGKY hierarchy \cite{CFTV,CHL,HT,MT}.
Let us define the $N-$particle density function 
\[
f^{(N)} = f^{(N)}(\theta,x_1,v_1,\dots,x_N,v_N,t),
\]
whose evolution, thanks to the mass conservation, is described in the terms of the Liouville equation
\be\label{eq:Liu}
\partial_t f^{(N)} + \sum_{i=1}^N v_i \cdot \nabla_{x_i}f^{(N)}= -\dfrac{1}{N} \sum_{i=1}^N \nabla_{v_i} \cdot \left( \sum_{j=1}^N H_{ij}(\theta)(v_j-v_i)f^{(N)} \right),
\ee
where we indicated with $H_{ij}(\theta)=H(\theta;x_i-x_j)$. Further, we define the marginal distribution 
\[\begin{split}
&f^{(1)}(\theta,x_1,v_1,t)=
\int_{\RR^{d_v(N-1)}}\int_{\RR^{d_x(N-1)}}f^{(N)}(\theta,x_1,v_1,x_{2,\dots,N},v_{2,\dots,N},t)dx_{2,\dots,N}dv_{2,\dots,N},
\end{split}\]
where 
\[
(x_{2,\dots,N},v_{2,\dots,N}) = (x_2,v_2,\dots,x_N,v_N).
\]
By direct integration of \eqref{eq:Liu} against $dx_{2,\dots,N},dv_{2,\dots,N}$ we have that  the transport term corresponds to
\[\begin{split}
\int_{\RR^{d_v(N-1)}}\int_{\RR^{d_x(N-1)}} \sum_{i=1}^N v_i\cdot  \nabla_{x_i}f^{(N)} dx_{2,\dots,N}&dv_{2,\dots,N} = v_1\cdot \nabla_{x_1}f^{(1)}(\theta,x_1,v_1,t).
\end{split}\]
For what it may concern the last term of \eqref{eq:Liu}, thanks to the interchangeability of the particles we have
\[\begin{split}
&\dfrac{1}{N} \sum_{i=1}^N  \int_{\RR^{d_v(N-1)}}\int_{\RR^{d_x(N-1)}}\sum_{j=1}^N \nabla_{v_i}H_{ij}(\theta)(v_j-v_i)f^{(N)} dx_{2,\dots,N}dv_{2,\dots,N}=\\
& \qquad \qquad \dfrac{1}{N}  \int_{\RR^{d_v(N-1)}}\int_{\RR^{d_x(N-1)}}\sum_{j=2}^N \nabla_{v_1}H_{1j}(\theta)(v_j-v_1)f^{(N)} dx_{2,\dots,N}dv_{2,\dots,N}.
\end{split}\]
By taking a closer look to this term we can observe how, thanks to the symmetry of the problem for all $2\le j,k\le N$, $j\ne k$ we have
\[\begin{split}
 &\int_{\RR^{d_v(N-1)}}\int_{\RR^{d_x(N-1)}} H_{1j}(\theta)(v_j-v_1)f^{(N)} dx_{2,\dots,N}dv_{2,\dots,N}=\\
& \qquad \qquad  \int_{\RR^{d_v(N-1)}}\int_{\RR^{d_x(N-1)}} H_{1k}(\theta)(v_k-v_1)f^{(N)} dx_{2,\dots,N}dv_{2,\dots,N},
\end{split}\]
and we obtain 
\be\begin{split}\label{eq:Liu_last2}
&\dfrac{1}{N} \sum_{i=1}^N  \int_{\RR^{d_v(N-1)}}\int_{\RR^{d_x(N-1)}}\sum_{j=1}^N \nabla_{v_i}H_{ij}(\theta)(v_j-v_i)f^{(N)} dx_{2,\dots,N}dv_{2,\dots,N}=\\
&\qquad \qquad \qquad\dfrac{N-1}{N}  \int_{\RR^{d_v(N-1)}}\int_{\RR^{d_x(N-1)}}  H_{12}(\theta)(v_2-v_1)f^{(N)} dx_{2,\dots,N}dv_{2,\dots,N}.
\end{split}\ee
Similarly to $f^{(1)}(\theta,x_1,v_1,t)$, let us define then the marginal density
\[
f^{(2)}(\theta,x_1,v_1,x_2,v_2,t) =  \int_{\RR^{d_v(N-2)}}\int_{\RR^{d_x(N-2)}}f^{(N)}dx_{3,\dots,N}dv_{3,\dots,N}.
\]
We can then reformulate \eqref{eq:Liu_last2} as 
\[
\dfrac{N-1}{N} \nabla_{v_1}\int_{\RR^{d_v}}\int_{\RR^{d_x}}H_{12}(\theta)(v_2-v_1)f^{(2)}dx_2dv_2. 
\]
Finally, the integration of \eqref{eq:Liu} against $dx_{2,\dots,N},dv_{2,\dots,N}$ gives
\be\begin{split}\label{eq:Liu3}
&\partial_t f^{(1)}(\theta,x_1,v_1,t) + v_1\cdot \nabla_{x_1} f^{(1)} = -\dfrac{N-1}{N} \int_{\RR^{d_v}}\int_{\RR^{d_x}}H_{12}(\theta)(v_2-v_1)f^{(2)}dx_2dv_2.
\end{split}\ee
Now, we define 
\[\begin{split}
&f(\theta,x_1,v_1,t) = \lim_{N\rightarrow +\infty}f^{(1)}(\theta,x_1,v_1,t),\\
&\tilde f(\theta,x_1,v_1,x_2,v_2,t) =  \lim_{N\rightarrow +\infty}f^{(2)}(\theta,x_1,v_1,x_2,v_2,t),  
\end{split}\]
and we make the ansatz for the propagation of chaos
\[
\tilde f(\theta,x_1,v_1,x_2,v_2,t) = f(\theta,x_1,v_1,t)f(\theta,x_2,v_2,t).
\]
Finally, from \eqref{eq:Liu3} we have
\be\label{eq:MF2}
\partial_t f(\theta,x,v,t) + v\cdot \nabla_x f(\theta,x,v,t) = \nabla_v\cdot \left[ \mathcal H[f](\theta,x,v,t)f(\theta,x,v,t) \right],
\ee
where
\be\label{eq:H_MF}
\mathcal H[f](\theta,x,v,t) = \int_{\RR^{d_v}}\int_{\RR^{d_x}} H(\theta;x,y)(v-w)f(\theta,y,w,t)dwdy,
\ee
where $H(\theta;x,y)=H(\theta;|x-y|)$ has been defined in \eqref{eq:HCS}. 

Alternative formal derivations require to assume that the set of particles remains in a given compact domain. Once this condition is met, thanks to the Prohorov's theorem, we can prove the convergence of the associated empirical distribution density $f^N$, up to extraction of a subsequence, to a continuous probability density, see \cite{CFTV} Section 3.2 for details. \\
 \\


\section{Monte Carlo gPC methods}
In this section we introduce the Stochastic Galerkin (SG) numerical method for a general differential problem. In particular we will discuss numerical methods belonging to the class of generalized polynomial chaos (gPC). Without intending to review all the pertinent literature we indicate the following references for an introduction \cite{JXZ,LeMK,X,XK}. In our schemes, Monte Carlo (MC) methods will be employed for the approximation of the distribution function $f(\theta,x,v,t)$ in the phase space whereas the random space at the particles level is approximated through SG-gPC techniques.

 For the sake of clarity, we recall first some basic notions on gPC approximation techniques and SG methods applied directly to the distribution function $f(\theta,x,v,t)$ and the corresponding mean-field system.
\subsection{Preliminaries on gPC techniques}
Let $(\Omega,\mathcal{F},P)$ be a probability space and let us define a random variable 
\begin{equation*}
\theta:(\Omega,\mathcal{F})\rightarrow (I_{\Theta},\mathcal{B}_{\RR}),
\end{equation*} 
with $I_{\Theta}\subseteq \RR$ and $\mathcal{B}_{\RR}$ the Borel set. Let us take into account moreover the space time domains $S\subseteq \RR^{d_x}\times \RR^{d_v},d\ge1$ and $[0,T]\in\RR^+$ respectively. In this short introduction we focus real-valued functions depending on a single random input of the form $f(\theta,x,v,t):\Omega \times S\times  [0,T]\rightarrow\RR^d$ with 
$
f(\cdot,x,v,t)\in L^2(\Omega,\mathcal{F},P)  \mbox{ for all } (x,v,t)\in S \times [0,T]\,.
$
We consider now the linear space $\mathbb{P}^M$ generated by orthogonal polynomials of $\theta$ with degree up to $M$: $\{\Phi_h(\theta)\}_{h=0}^M$. They form an orthogonal basis of $L^2(\Omega,\mathcal{F},P)$
\begin{equation*}
\int_{I_{\Theta}}\Phi_h(\theta)\Phi_k(\theta)dg(\theta)=\int_{I_{\Theta}}\Phi_h^2(\theta)dg(\theta)\delta_{hk}
\end{equation*}
where $\delta_{hk}$ is the Kronecker delta function and $g(\theta)$ is the probability distribution function of the random variable $\theta\in I_{\Theta}$. Let us assume that $g(\theta)$ has finite second order moment, then the polynomial chaos expansion of $f(\cdot,\cdot,\cdot)$ is defined as follows
\begin{equation*}
f(\theta,x,v,t)=\sum_{m\in\mathbb{N}}\hat{f}_m(x,v,t)\Phi_m(\theta),
\end{equation*}
where $\hat{f}_m(x,v,t)$ is the Galerkin projection of $f(\theta,x,v,t)$ into the polynomial space $\mathbb{P}^m$
\begin{equation}\label{eq:f_pro}
\hat{f}_m(x,v,t)=\int_{I_{\Theta}}f(\theta,x,v,t)\Phi_m(\theta)dg(\theta),\qquad m\in\mathbb{N}.
\end{equation}

We exemplify the effective numerical method on a general nonlinear initial value problem
\begin{equation}\label{eq:differential_problem}
\partial_t f(\theta,x,v,t)=\mathcal{J}[f](\theta,x,v,t)
\end{equation}
with $f(\theta,x,v,t)$ solution of the introduced differential model and $\mathcal J[\cdot]$ a differential operator. Here the random variable $\theta$ acts as a perturbation of $\mathcal{J}[\cdot]$, of the solution or propagates from uncertain initial conditions. 

\begin{table}[t]
\begin{center}
\begin{tabular}{c | c |  c}\label{tab:Ask}
Probability law of $\theta$ & Expansion polynomials & Support \\ \hline\hline
Gaussian & Hermite    & $(-\infty,+\infty)$ \\
Uniform    & Legendre & $[a,b]$\\
Beta         & Jacobi     & $[a,b]$\\
Gamma    & Laguerre & $[0,+\infty)$\\
Poisson    & Charlier   & $\mathbb{N}$\\
\end{tabular}
\end{center}
\caption{The different gPC choices for the polynomial expansions}
\label{tab:pol}
\end{table}

The generalized polynomial chaos method approximates the solution $f(\theta,x,v,t)$ of \eqref{eq:differential_problem} with its $M$th order truncation $f^M(\theta,x,v,t)$ and considers the Galerkin projections of problem for each $h=0,\dots,M$
\begin{equation*}
\partial_t \int_{I_{\Theta}} f(\theta,x,v,t)\cdot \Phi_h(\theta)dg(\theta)=\int_{I_{\Theta}}\mathcal{J}[f](\theta,x,v,t)\cdot \Phi_h(\theta)dg(\theta)
\end{equation*} 
Thanks to the Galerkin orthogonality we typically obtain a system of $M+1$ deterministic coupled equations  
\begin{equation}\label{eq:system_det}
\hat{f}_h(x,v,t) = \hat{\mathcal J}_h(\{\hat{f}_k\}_{k=0}^M)(x,v,t).
\end{equation}
The related deterministic coupled subproblem can be solved through suitable numerical techniques. The approximation of the statistical quantities of interest are defined in terms of the introduced projections. From \eqref{eq:f_pro} we have
\be\label{eq:mean_teo}
\mathbb E[f(\theta,x,v,t)] \approx \hat f_0(x,v,t),
\ee
and its evolution is approximated by \eqref{eq:system_det}. Thanks to the orthogonality of the polynomial basis it is possible to show that 
\be\label{eq:var_teo}
\begin{split}
Var[f(\theta,x,v,t)] &\approx \mathbb E\left[ \left(\sum_{h=0}^M \hat f_h(x,v,t) \Phi_h(\theta)-\hat f_0(x,v,t)\right)^2 \right]\\
&= \sum_{h=0}^M \hat f_h(x,v,t)\mathbb E[\Phi_h^2(\theta)]-\hat f_0^2(x,v,t).
\end{split}
\ee

One of the most important advantages of the gPC-SG type methods is their exponential convergence with respect to the stochastic quantity to the exact solution of the problem, unlike usual sampling techniques for which the order is $\mathcal O(1/\sqrt{M})$ where $M$ is the number of samples. Despite this property, numerical solution of gPC-SG systems are costly in the case of nonlinear problems, given the absence of parallelization, and for this reason requires a clear effort in order to design efficient codes \cite{LeMK,X}.

\subsection{Stochastic Galerkin methods for the mean-field system}
Let us consider the stochastic mean-field equation \eqref{eq:MF2} with nonlocal drift $\mathcal H[\cdot]$ of the form \eqref{eq:H_MF}. The gPC approximation for this problem is given by the following system of differential equations
 \be\label{eq:CSgPC}
 \partial_t \hat f_h(x,v,t) + v\cdot \nabla_x \hat f_h(x,v,t) = \nabla_v \cdot \Big[\sum_{k=0}^M \mathcal H_{hk}[\hat f](x,v,t)\hat f_k(x,v,t) \Big],
 \ee
 with 
 \be\label{eq:HgPC}
 \mathcal H_{hk}[\hat f] = \dfrac{1}{\|\Phi_h \|_{L^2}^2}\sum_{m=0}^M\int_{I_{\Theta}}\mathcal H[\hat f_m]\Phi_k(\theta)\Phi_m(\theta)\Phi_h(\theta)dg(\theta).
 \ee
 The system of differential equations \eqref{eq:CSgPC} may be written in vector notations as follows
 \[
 \partial_t \hat{\mathbf{f}}(x,v,t) + v\cdot \nabla_x \hat{\mathbf{f}}(x,v,t) = \nabla_v \cdot \Big[ \mathbf{H}[\hat{\mathbf{f}}](x,v,t)\hat{\mathbf{f}}(x,v,t) \Big],
 \]
 where $\hat{\mathbf{f}}= (\hat{f}_0,\dots,\hat{f}_M)^T$ and the components of the $(M+1)\times(M+1)$ matrix $\mathbf{H}[\hat{\mathbf{f}}]$ are given by \eqref{eq:HgPC}. We define the total mass and the mean velocity as the quantities:
 \[
 \rho = \int_{\RR^{d_v}\times \RR^{d_x}} f^M(\theta,x,v,t)dxdv\qquad \mbox{and} \qquad
 V(\theta,t) =\dfrac{1}{\rho} \int_{\RR^{d_v}\times \RR^{d_x}}vf^M(\theta,x,v,t)dxdv\,,
 \]
with
$$
f^M(\theta,x,v,t)=\sum_{m=0}^M \hat{f}_m(x,v,t)\Phi_m(\theta)\,.
$$
 
\begin{prop}
The total mass $\rho$ does not depend on the stochastic quantity $\theta\in I_{\Theta}$ in the case of deterministic initial distribution $f(\theta,x,v,0)=f_0(x,v)$ and is conserved in time under the conditions
$$
  v\hat f_h(x,v,t)\Big |_{\RR^d} = 0  \qquad \mbox{and} \qquad  \sum_{k=0}^M \mathcal H_{hk}[\hat{f}]\hat{f}_k(x,v,t)\Big|_{\RR^d} = 0
$$
 for all $h=0,\dots,M$.
 \end{prop}
 
\begin{proof}
Integrating \eqref{eq:CSgPC} in phase space leads to
$$
  \partial_t \int_{\RR^{2d}}\hat f_h(x,v,t)dxdv=0
$$
under the above assumptions. Therefore, the total mass is conserved in time and since the  initial data does not depend on $\theta$, then $\hat f_h =0$ for $h\neq 0$ and the result follows.
\end{proof}
 
 \begin{prop}
 The mean velocity of \eqref{eq:CSgPC}-\eqref{eq:HgPC}
 is conserved in time provided
 \be\label{eq:cond_v}
 v\otimes v \hat{f}_h(x,v,t)\Big|_{\RR^d}=0 \qquad \mbox{and} \qquad v \Big[\sum_{k=0}^M \mathcal H_{hk}[\hat f]\hat{f}_k(x,v,t) \Big] \Big|_{\RR^d}=0\,.
 \ee
 \end{prop}
 
 \begin{proof}
 For all $h=0,\dots,M$ let us consider the quantity $ \int_{\RR^{2d}}v\hat f_h(x,v,t)dxdv$. From \eqref{eq:CSgPC} we have
 \[
 \begin{split}
 \partial_t \int_{\RR^{2d}}v\hat f_h(x,v,t)dxdv +& \int_{\RR^{2d}}\nabla_x \cdot (v\otimes v \hat{f}_h(x,v,t)) dxdv = \\
 &\qquad \int_{\RR^{2d}}v\nabla_v\cdot \Big[ \sum_{k=0}^M \mathcal H_{hk}[\hat f]\hat{f}_k(x,v,t) \Big]dxdv.
 \end{split}\]
 Thanks to \eqref{eq:cond_v} we have
 \[
 \partial_t \int_{\RR^{2d}}v\hat f_h(x,v,t)dxdv = -\sum_{k=0}^M \int_{\RR^{2d}}\mathcal H_{hk}[\hat f]\hat{f}_k(x,v,t)dxdv,
 \]
 from the definition of $\mathcal H_{hk}[\cdot]$, it follows that
 \[
 \begin{split}
 &\sum_{m,k=0}^M \int_{\RR^{2d}}\int_{\RR^{2d}}H(x,y;\theta)v\hat f_m(y,w,t)dwdy\hat f_k(x,v,t)dxdv = \\
&\qquad  \sum_{m,k=0}^M \int_{\RR^{2d}}\int_{\RR^{2d}}H(y,x;\theta)w\hat f_m(x,v,t)dxdv\hat f_k(y,w,t)dydw,
\end{split}
 \]
implying that
$$
\sum_{m,k=0}^M \mathcal H_{hk}[\hat f]\hat{f}_k(x,v,t)=0
$$
due to the symmetry of the interaction function $H(\cdot,\cdot;\theta)$ for all $\theta$. Being $f^M(\theta,x,v,t)=\sum_{m=0}^M \hat f_m(x,v,t)\Phi_m(\theta)$, the result follows.
\end{proof}
 
 \subsection{Monte Carlo gPC scheme}
 Similarly to classical spectral methods, the solution of the coupled system $f^M$ looses its positivity and so it looses a clear physical meaning. This fact represents a serious drawback for real world applications of these expansions for which positivity of statistical quantities, like the expected solution, is necessary. In order to overcome this difficulty, we construct an effective numerical method for the solution of the mean-field stochastic equations of collective behavior having roots in Monte Carlo methods (see \cite{PR_ESAIM,PT} for an introduction). In particular we employ the results reported in Section \ref{sec:MF}, where we formally derived the mean-field description of an interacting system of agents from a microscopic stochastic dynamics. The proposed method for mean-field equations of collective behavior is capable to efficiently approximate statistical quantities of the system \eqref{eq:mean_teo}-\eqref{eq:var_teo} and to conserve their positivity. 

\subsubsection{Stochastic Galerkin methods for the particle system}
Similarly to what we described for the mean-field equations we can consider the gPC approximation of the microscopic dynamics. We approximate the position and the velocity of the $i$th agent as follows
\begin{equation*}
x_i(\theta,t) \approx x_i^M = \displaystyle\sum_{k=0}^M \hat{x}_{i,k}(t)\Phi_k(\theta), \qquad v_i(\theta,t) \approx v_i^M = \displaystyle\sum_{k=0}^M \hat{v}_{i,k}\Phi_k(\theta),
\end{equation*}
where 
\begin{equation*}
\hat x_{i,k} = \int_{I_{\Theta}}x_i(\theta,t)\Phi_k(\theta)dg(\theta),\qquad \hat v_{i,k}=\int_{I_{\Theta}}v_i(\theta,t)\Phi_k(\theta)dg(\theta).
\end{equation*}
We obtain the following polynomial chaos expansion for all $h=0,\dots,M$
\begin{equation*}
\begin{cases}
\dfrac{d}{dt} \hat{x}_{i,h}(t)  &= \hat{v}_{i,h}(t),\\
\dfrac{d}{dt} \hat{v}_{i,h}(t)  &= \displaystyle\dfrac{1}{N}\sum_{j=1}^N \sum_{k=0}^M e^{ij}_{hk}(\hat v_{j,k}(t)-\hat v_{i,k}(t)) 
\end{cases}
\end{equation*}
where 
\begin{equation*}
e^{ij}_{hk} =\dfrac{1}{\|\Phi_h(\theta) \|^2} \int_{I_{\Theta}} H(\theta;x_i,x_j)\Phi_k(\theta)\Phi_h(\theta) dg(\theta),
\end{equation*}
defines a time-dependent matrix $\mathcal E = \left[e_{hk}^{ij}\right]_{h,k=0,\dots,M}$. At the microscopic level the conservation of the mean velocity holds thanks again to the symmetry of the interaction function $H(\cdot,\cdot;\theta)$ introduced in \eqref{eq:HCS}.
In fact, the gPC approximation $v^M(\theta,t)$ also conserves the mean velocity as proven in \cite{APZa}.

\subsubsection{Monte Carlo-gPC approximations }
Let us finally tackle the limiting stochastic mean-field equation in its gPC approximation. As already observed the particle solution of \eqref{eq:CSgPC}-\eqref{eq:HgPC} corresponds to compute the original $O(MN^2)$ dynamics, since at each time step and for each gPC mode every agent averages its velocity with the projected velocities of the whole set of agents. A reduction in computational cost may be achieved through a Monte Carlo (MC) evaluation of the interaction step as introduced in \cite{AlbiPareschi13}. Once we have an effective MC algorithm for transport and interaction in phase space, the expected solution may be reconstructed from expected positions and velocities of the microscopic system, which has been computed in the gPC setting. 

\begin{alg}[\textbf{MC-gPC for stochastic mean-field equations}]
\qquad\qquad\qquad\qquad
\begin{itemize}
\item[1.] Consider $N$ samples $(x_{i},v_{i})$ with $i=1,\dots,N$ from the initial $f_0(x,v)$, and fix $S\le N$ a positive integer;
\item[2.] \textit{for $n=0$ to $T-1$}
\begin{itemize}
\item [] \hspace{-0.7cm}\textit{for $i = 1$ to $N$}
\item [a)] sample $S$ particles $j_1,\dots,j_S$ uniformly without repetition among all particles;
\item[b)] perform gPC up to order $M\ge 0$ over the set of $S\le N$ particles: we need to compute the dynamics for $(\hat x_{j_s,h},\hat v_{j_s,h})$ where $s = 1,\dots,S$ and $h =0,\dots,M$.
\item[c)] compute the position and velocity change 
\[\begin{split}
\hat v_{i,h}^{n+1} = \hat v_{i,h} + \dfrac{\Delta t}{S}\sum_{s=1}^S\sum_{k=0}^M e_{kh}^{i{j_s}}(\hat v_{j_s,k}-\hat v_{i,k})
\end{split}\]
\item[]  \hspace{-0.7cm}\textit{end for}
\end{itemize}
\item[3.] Reconstruction $\mathbb E_{\theta}[f(x,v,\theta,n\Delta t)]$ .
\item[]\textit{end for}
\end{itemize}
\label{alg:1}
\end{alg}

\begin{figure}
\centering
\includegraphics[scale=1.0]{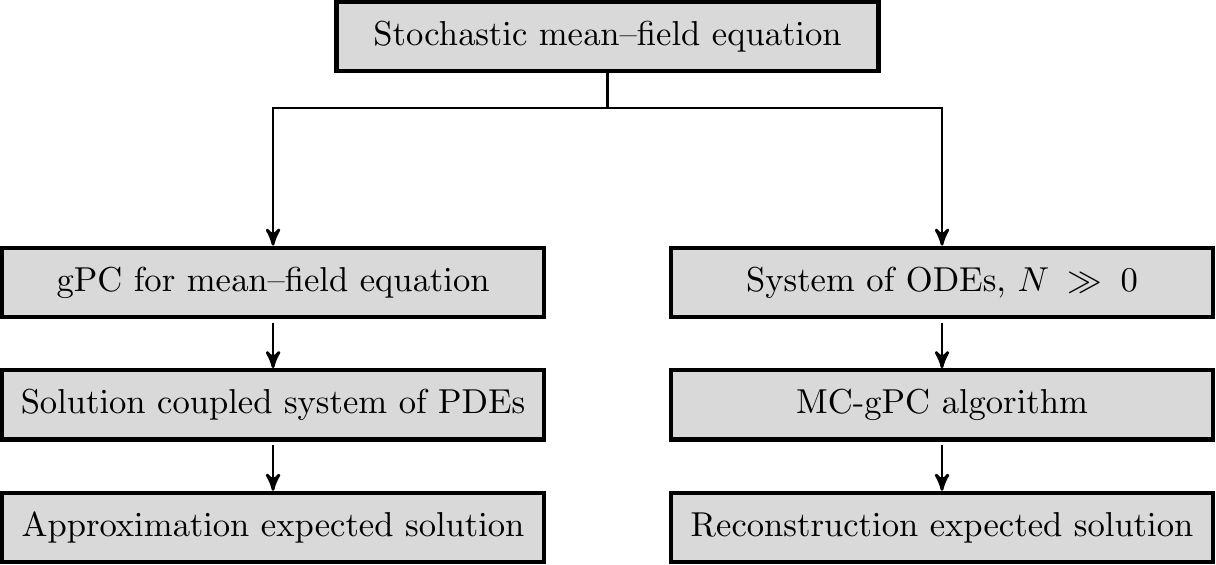}
\caption{Two possible numerical approaches to stochastic mean-field models, the right branch describe the MC-gPC scheme.}
\label{fig:scheme}
\end{figure}

The Monte Carlo evaluation of the interaction step described in Algorithm \ref{alg:1} allows to reduce the overall cost to $O(MSN)$, where $S\ll N$.  Of course, in the case $S=N$, we obtain the original cost of the $N$ particles system. Clearly, the introduced method is still spectrally accurate with respect to the stochastic variable $\theta$ provided we have a smooth dependence of the particle solution from the random field.

In Figure \ref{fig:scheme} we sketch the two possible approaches to numerical solution of stochastic mean-field equations. On the left, we first consider the gPC approximation of the original PDE of the mean-field type, then through a deterministic solver we tackle the coupled system of equations in order to approximate the expected quantities. On the right we consider the MC-gPC scheme, therefore we work on the particle system thanks to a Monte Carlo evaluation of the interactions, then by considering the gPC scheme at the microscopic level we can reconstruct statistical quantities. 
Several approaches are possible when reconstructing densities from particles, in the present manuscript we consider the histogram of position and velocity of the set of particles in the phase space. Other approaches are the so-called weighted area rule \cite{HE}, where each particle is counted in a computational cell and is counted in the neighbor cells with a fraction proportional to the overlapping area, or reconstruct the density function by a convolution of the empirical particle distribution with a suitable mollifier \cite{PR_ESAIM}. Clearly, the resulting method preserves the positivity of the distribution function.\\
\begin{rem}   
Concerning the computational accuracy of the method in the case $S=N$, we have a convergence rate of the order $O(1/\sqrt{N})$, where $N$ is the number of samples, in the physical space due to the Monte Carlo approximation and a spectral convergence with respect to $M$ in the random space.  In the case $S < N$ the fast evaluation of the interaction sum with $S$ points, in a single time step, introduces an additional error $O(\sqrt{1/S-1/N})$ at the particles level. Therefore, in practical applications very few modes are necessary to match the accuracy of the Monte Carlo solver. In particular, we will show how, for a fixed number of particles $N$, macroscopic expected quantities are approximated with spectral accuracy, typical of SG methods.  
\end{rem}


\section{Applications}

In this section we present numerical tests based on stochastic mean-field equation of collective behavior. In particular, we give numerical evidence of the effectiveness of MC-gPC methods showing that the method does not loose the spectral convergence of gPC methods when approximating the expected solution of the system and that preserves the positivity of the statistical quantities. In all tests the time integration has been performed through a 4$th$ order Runge-Kutta method. 

\subsection{1D Tests}

\paragraph{The space homogeneous case} 
Let us consider the space independent case in the one dimensional setting, i.e. $H(\theta;x_i,x_j)=K(\theta)>0$, for all $\theta$. The evolution of the density function $f(\theta,v,t)$, $v\in\RR$ is given by the following stochastic mean-field equation 
\be\label{eq:hom_MF}
\partial_t f(\theta,v,t) = \partial_v \Big[K(\theta)(v-u)f(\theta,v,t) \Big],
\ee
with $u = \int_{I_{\Theta}}vf_0(v)dv$, whose gPC approximation is given for all $h=0,\dots,M$ by
\be\label{eq:hom_MFgPC}
\partial_t \hat f_h(v,t) = \dfrac{1}{\| \Phi_h \|^2} \partial_v \left[ \sum_{k=0}^M(v-u) \mathcal H_{hk}\hat f_k(v,t)\right],
\ee
where now
\[
\mathcal H_{hk} = \int_{I_{\Theta}}K(\theta)\Phi_h(\theta)\Phi_k(\theta)dg(\theta),\qquad  \hat f_h(v,t) = \int_{I_{\Theta}}f(\theta,v,t)\Phi_h(\theta)dg(\theta).
\]

\begin{figure}[t]
\centering
\includegraphics[scale=0.5]{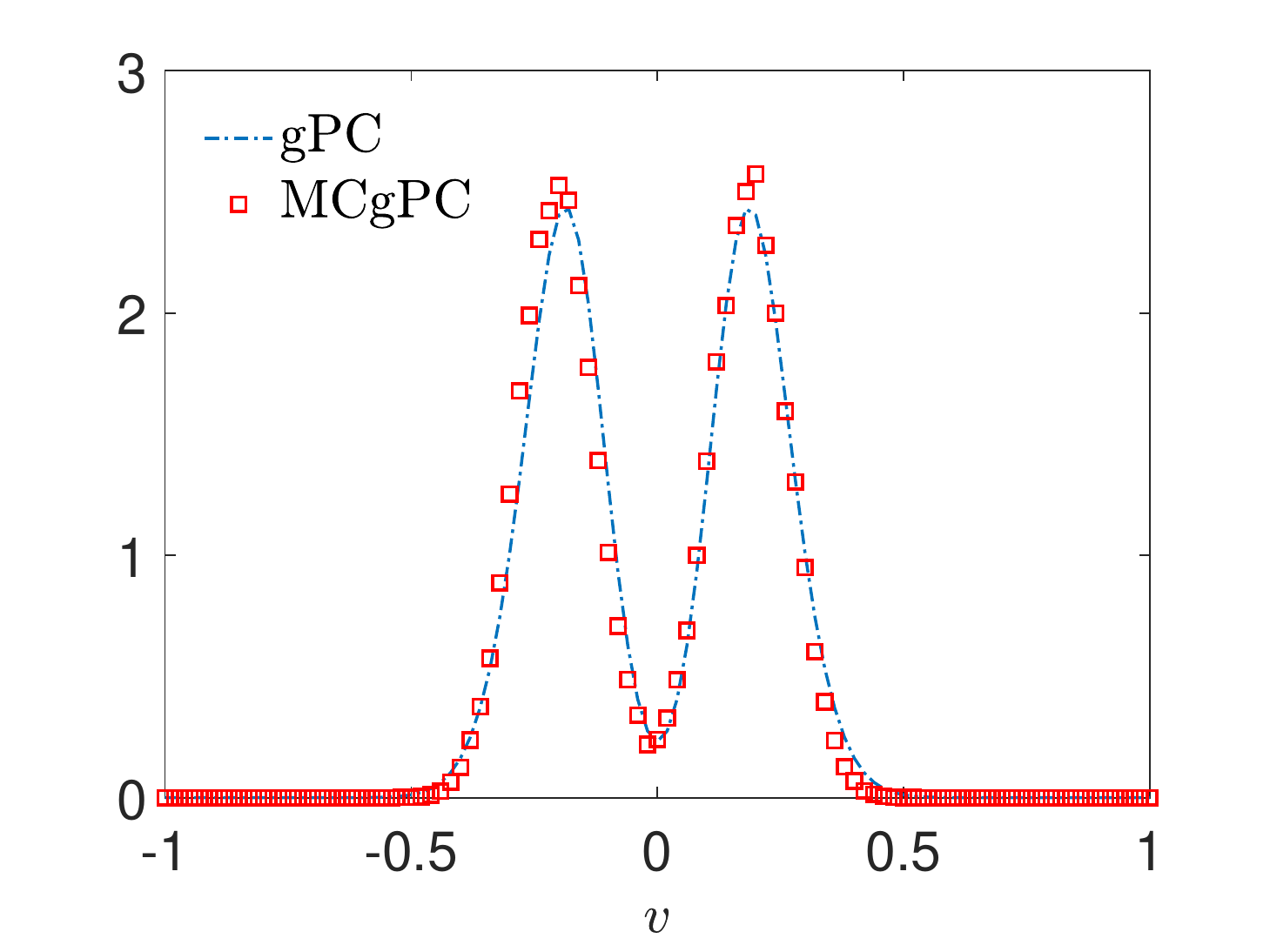}
\includegraphics[scale=0.5]{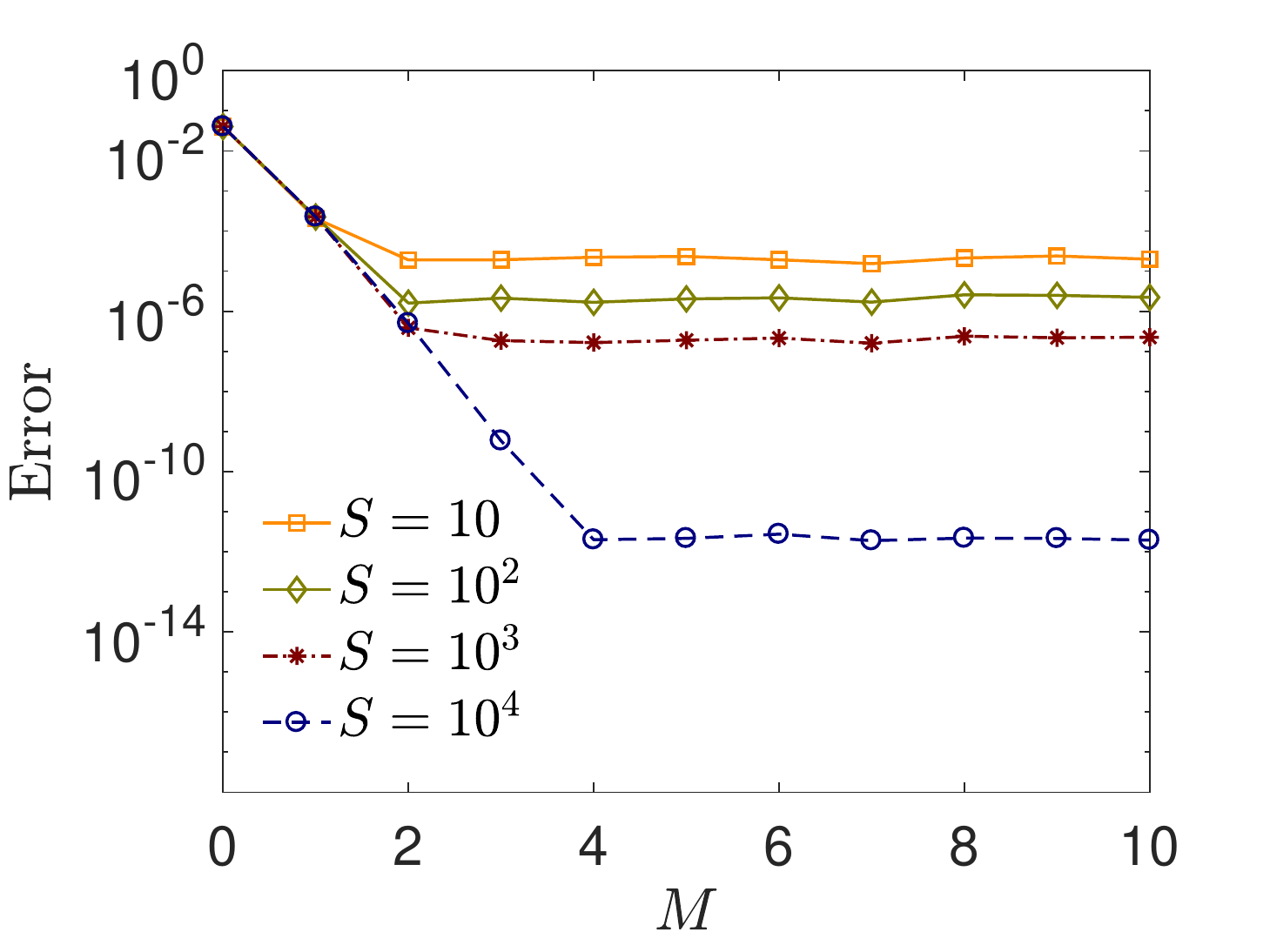}
\caption{\textbf{Left}: expected density at time $T=1$ obtained from \eqref{eq:hom_MFgPC} through the gPC and the one obtained through MCgPC schemes with a sampling of $S=5$ at each time step. We considered $101$ gridpoints in the velocity space, $\Delta t=\Delta v^2$, the gPC expansion has been performed up to order $M = 5$. \textbf{Right}: convergence of the MCgPC algorithm based on a reference expected temperature $\mathcal T^{ref}$ at time $T=1$ calculated at the particle level with $10^5$ particles. }
\label{fig:1}
\end{figure}

\begin{figure}[ht!]
\centering
\includegraphics[scale=0.5]{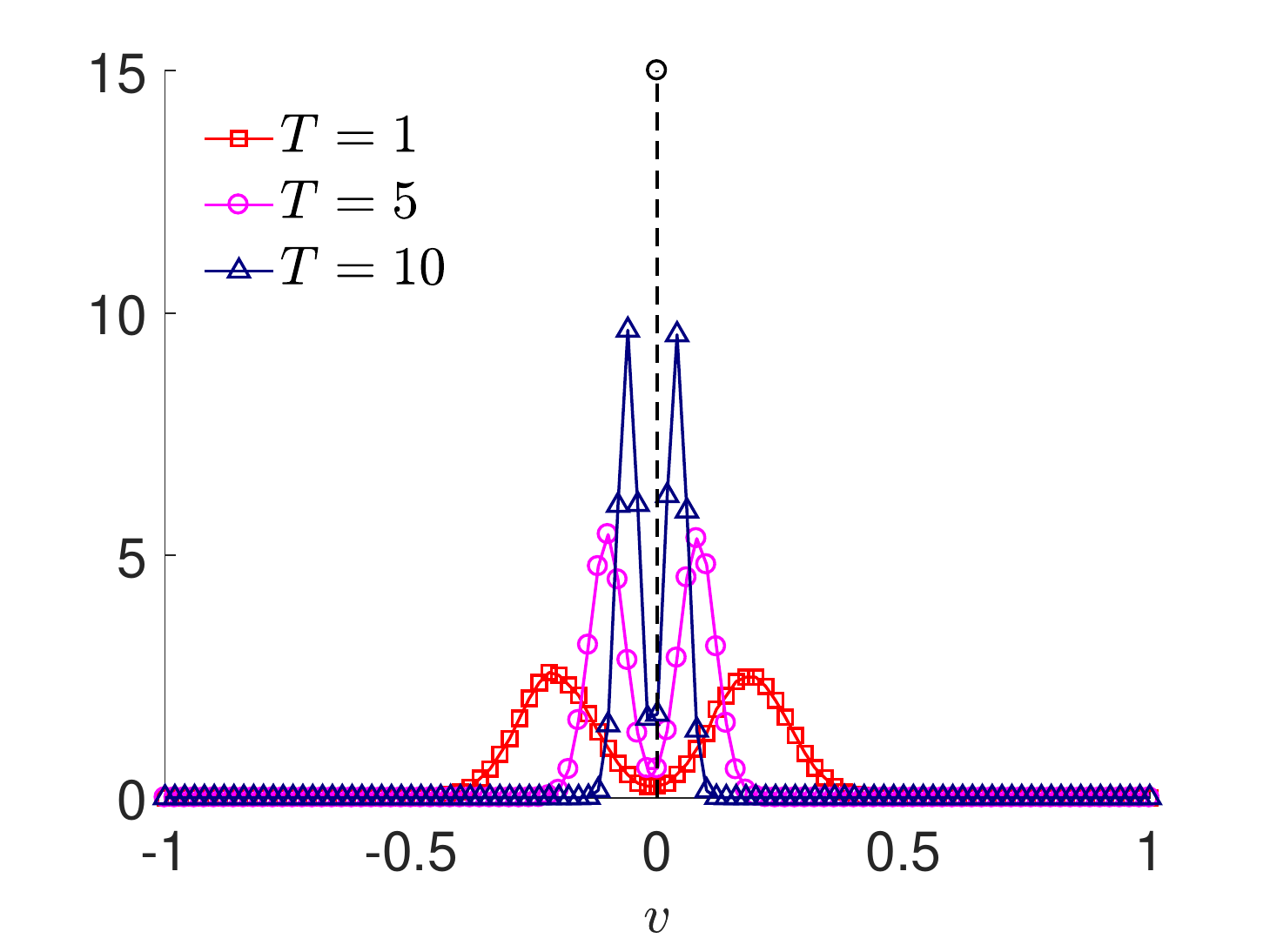}
\includegraphics[scale=0.5]{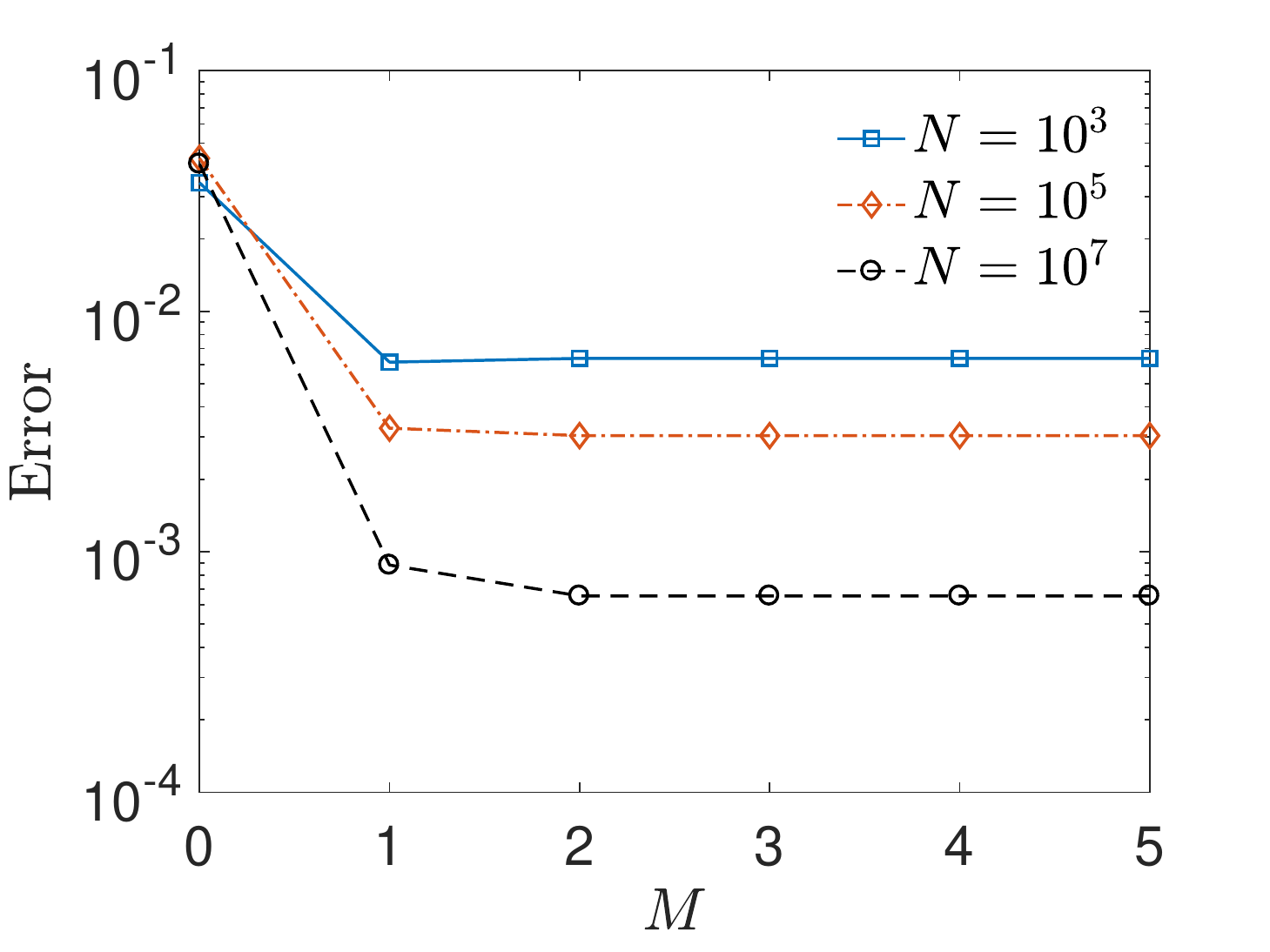}
\caption{\textbf{Left}: evolution of the expected density toward the Dirac delta $\delta(v-u)$, $u=0$ for the MCgPC scheme at different times with $S=50$ uniformly sampled points at each time step. \textbf{Right}: mean-field convergence of the MCgPC algorithm with fixed $S=50$ and an increasing number of particles $N$, we compare the obtained expected temperature with a reference one $\mathcal T^{ref}$ at time $T=1$ which is computed from the mean-field problem evolved with $801$ gridpoints. }
\label{fig:2}
\end{figure}

The long time solution of \eqref{eq:hom_MF} is a Dirac delta $\delta(v-u)$ provided $K(\theta)>0$ for all $\theta$, see \cite{CFRT,TZ}. We compute the transient behavior of the gPC coupled system of homogeneous equations \eqref{eq:hom_MFgPC} through a central difference scheme. 
Let us consider an initial density function $f_0(v)$ of the form 
\be\label{eq:f0}
f_0(v) = \beta \left[\exp(-\dfrac{(v-\mu)^2}{2\sigma^2})+\exp(-\dfrac{(v+\mu)^2}{2\sigma^2})\right], \qquad\sigma^2 = 0.1, \mu = \dfrac{1}{4},
\ee
with $\beta>0$ a normalization constant. Discrete samples of the initial velocities of the microscopic system of ODEs are obtained from $f_0(v)$ in \eqref{eq:f0}.

In Figures \ref{fig:1}-\ref{fig:2} we study the convergence of the expected temperature of the system
$$
\mathcal T =\int_{I_{\Theta}} \int_{\RR} (v-u)^2  f(\theta,v,t)dv dg(\theta)
$$
obtained through the MCgPC algorithm for an increasing order of the gPC expansion. We considered $N=10^4$ particles computing the evolution up to time $T=1.0$ with time step $\Delta t=10^{-2}$, the MCgPC method considered an increasing number of interacting particles $S=10,10^2,10^4$. In particular in Figure \ref{fig:2}, we consider as reference temperature the one obtained at the mean-field level with a high number of gridpoints at time $T=1$, whereas in Figure \ref{fig:1} as reference temperature we consider the one obtained with $N=10^5$ particles interacting with the full set, i.e. $S=10^5$, at time $T=1$. 

We can observe in Figure \ref{fig:1} (right) that for this test case the error with respect to $S$, for a given $M \geq 4$, decays as $O(1/S-1/N)$ instead of $O(\sqrt{1/S-1/N})$. This is due to the fact that we are evaluating the error for the temperature and that the mean velocity is zero.

\begin{figure}[ht!]
\centering
\subfigure[$t=0$]{\includegraphics[scale=0.410]{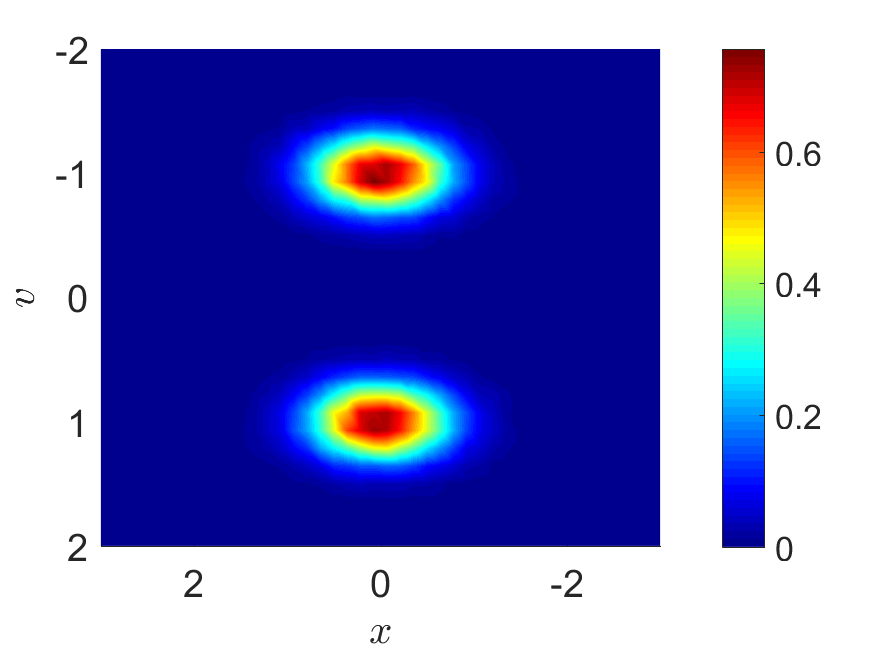}}
\subfigure[$t=1$]{\includegraphics[scale=0.410]{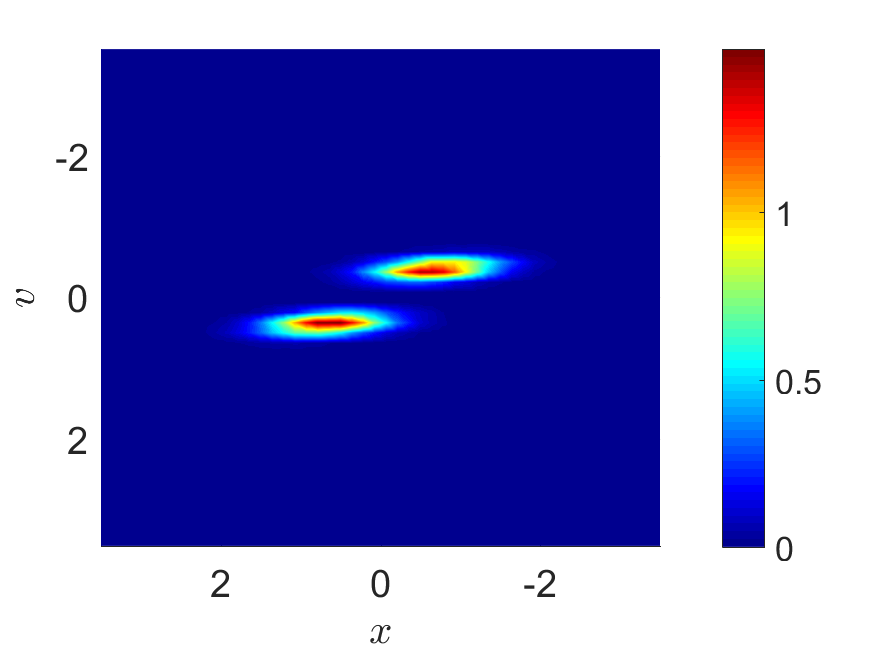}}\\
\subfigure[$t=3$]{\includegraphics[scale=0.410]{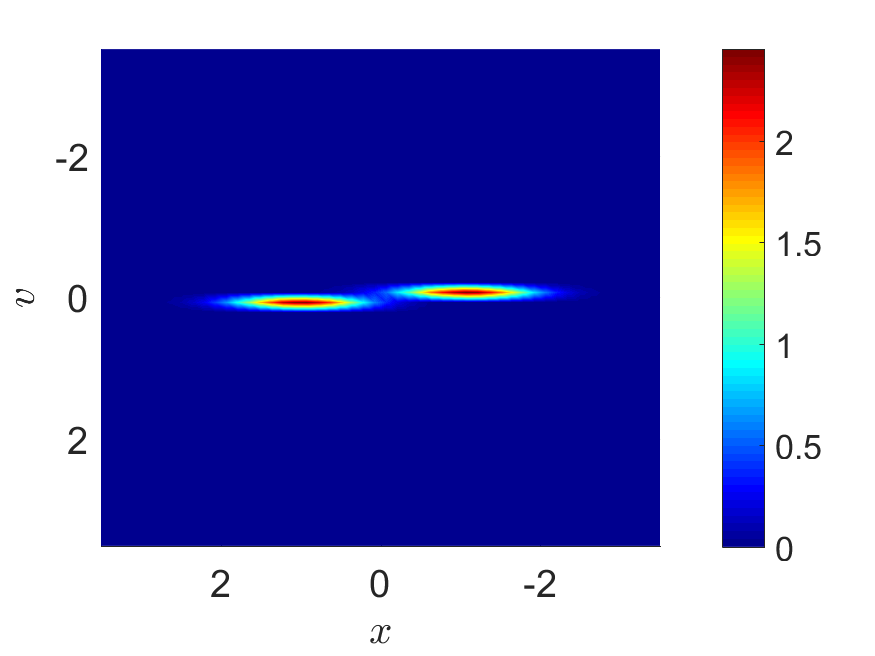}}
\subfigure[$t=5$]{\includegraphics[scale=0.410]{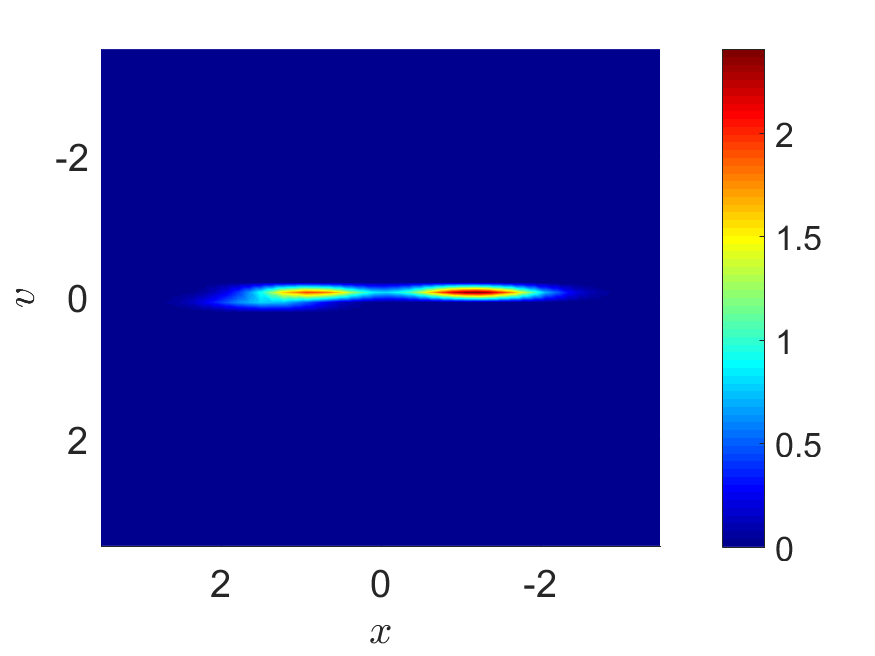}}
\caption{1D Cucker-Smale dynamics computed through the MCgPC algorithm over the time interval $t\in[0,5]$, $\Delta t=10^{-2}$. We considered $N=10^5$ agents, $S=5$, and a stochastic Galerkin decomposition up to order $M=5$. Stochastic interactions are given by $H(x_i,x_j;\theta)$ with $\gamma(\theta)=0.05+0.05\theta$, $\theta\sim U([-1,1])$.  }
\label{fig:CS_1D}
\end{figure}
\vspace{-0.5cm}
\paragraph{Stochastic 1D Cucker-Smale dynamics}
In this test, we consider the 1D Cucker-Smale dynamics. Let us consider as initial distribution the following bivariate and bimodal distribution
\[
f_0(x,v) = \dfrac{1}{2\pi\sigma_x\sigma_v} \exp\Big(-\dfrac{x^2}{2\sigma_x^2}\Big) \Big[ \exp\Big(-\dfrac{(v+\bar v)^2}{2\sigma_v^2}\Big)+\exp\Big(-\dfrac{(v+\bar v)^2}{2\sigma_v^2}\Big) \Big],
\]
with $\bar v = 1$, $\sigma_x^2=0.5$, $\sigma_v^2=0.2$. Our initial data for particle postions and velocities are sample from $f_0$.

In Figure \ref{fig:CS_1D} we present the evolution over the time interval $[0,5]$ of the expected distribution following the stochastic Cucker-Smale dynamics with stochastic interactions. The results are obtained through the MCgPC scheme with $N=10^5$ particles, whose interactions are calculated over subsets of $S=5$ particles. The stochastic interactions are given by
\be\label{eq:num_CS1D}
H(\theta;|x_i^M-x_j^M|)=\dfrac{1}{(1+|x_i^M-x_j^M|^2)^{\gamma(\theta)}},
\ee
with $\gamma(\theta) = 0.1 + 0.05\theta$ a stochastic quantity depending on $\theta\sim U([-1,1])$. 

We considered an $M=5$ order stochastic Galerkin method at the microscopic level. The expected distribution is then reconstructed in the domain $[-2,2]\times [-2,2]$ discretized with $50$ gridpoints both in space and velocity. 

\begin{figure}[ht!]
\centering
\subfigure[$t=0$]{\includegraphics[scale=0.420]{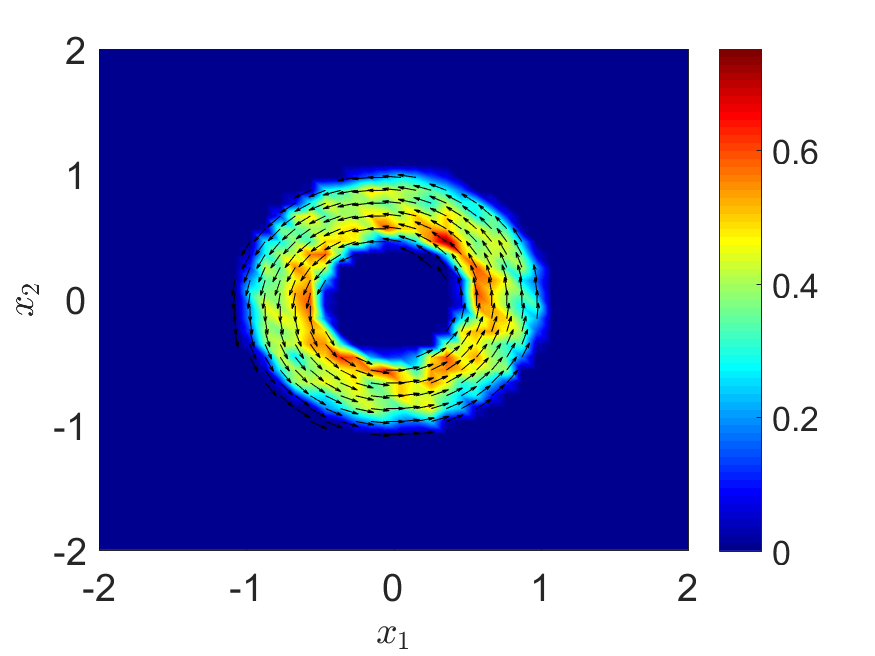}}
\subfigure[$t=2$]{\includegraphics[scale=0.420]{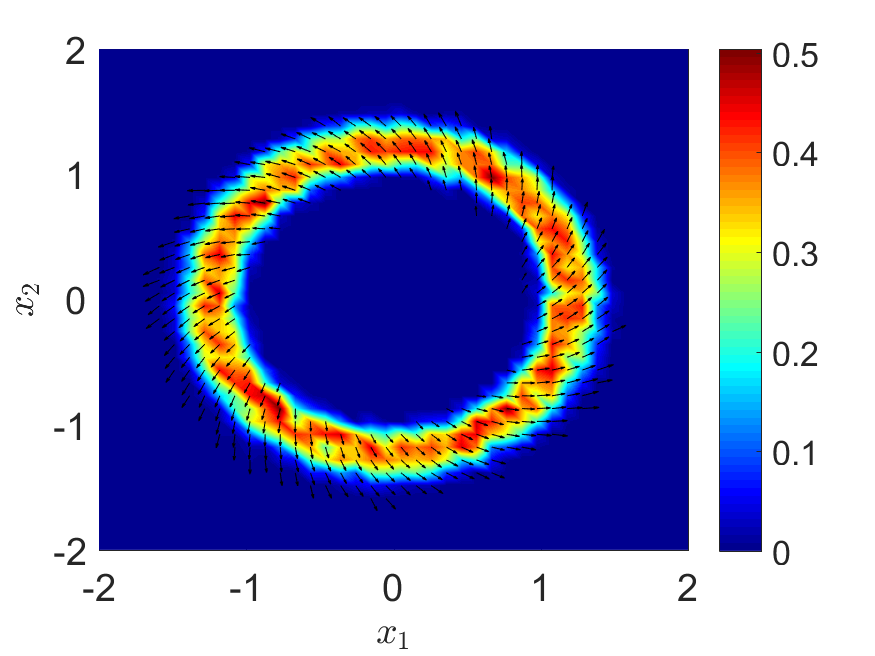}}\\
\subfigure[$t=8$]{\includegraphics[scale=0.420]{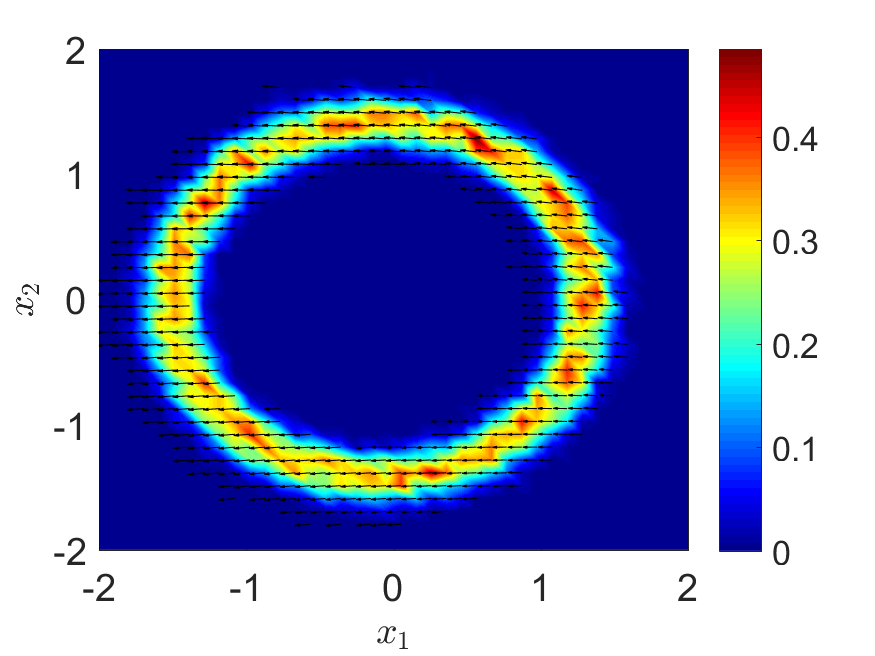}}
\subfigure[$t=10$]{\includegraphics[scale=0.420]{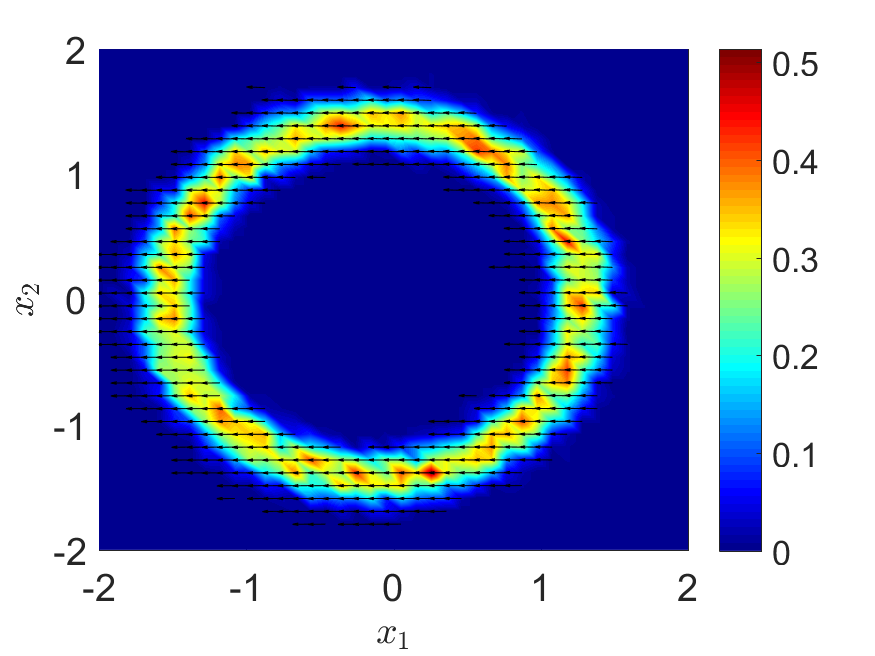}}
\caption{2D Cucker-Smale with $N=10^5$ agents, $H(\theta;|x_i-x_j|)$ with $\gamma(\theta)=0.1+0.05\theta$, $\theta\sim \mathcal U([-1,1])$.  
2D Cucker-Smale dynamics computed through the MCgPC algorithm over the time interval $t\in[0,10]$, $\Delta t=10^{-2}$. We considered $N=10^5$ agents, $S=5$, and a stochastic Galerkin decomposition of order $M=10$. Stochastic interactions are given by \eqref{eq:num_CS1D}. }
\label{fig:4}
\end{figure}

\subsection{2D Tests}

\paragraph{Stochastic 2D Cucker-Smale.}

In Figure \ref{fig:4} we computed the evolution over the time interval $t\in[0,10]$ of 2D Cucker-Smale non homogeneous mean-field model in the space domain $[-2,2]\times [-2,2]$ through the MCgPC scheme. As initial distribution we considered uniformly distributed $N=10^5$ particles on a 2D annulus with a circular counterclockwise motion 
\[
f_0(x,v) = \dfrac{1}{|C|} \chi(x\in C) \;\delta\left(v-\dfrac{k\wedge x}{|x|}\right),
\]
being $C:=\{ x\in\RR^2: 0.5 \le |x_1-x_2|\le 1\}$ and $ k$ the fundamental unit vector of the $z$-axis. 

Similarly to the 1D case we considered stochastic interactions described by \eqref{eq:num_CS1D}. The mean-field Monte Carlo step has been considered with $S=5$ interacting particles. The stochastic Galerkin projection has been considered up to order $M=10$. 

The evolution shows how the initial distribution flocks exponentially fast and that at time $t=8$ the final flocking structure is essentially reached. The reconstruction step of the mean density for position and velocity has been done with $50$ gridpoints in both space dimensions. In order to highlight the flocking formation in each figure we also add the velocity field (arrows in the plots) to illustrate the flock direction. 

\begin{figure}
\centering
\subfigure[$t=10$]{\includegraphics[scale=0.420]{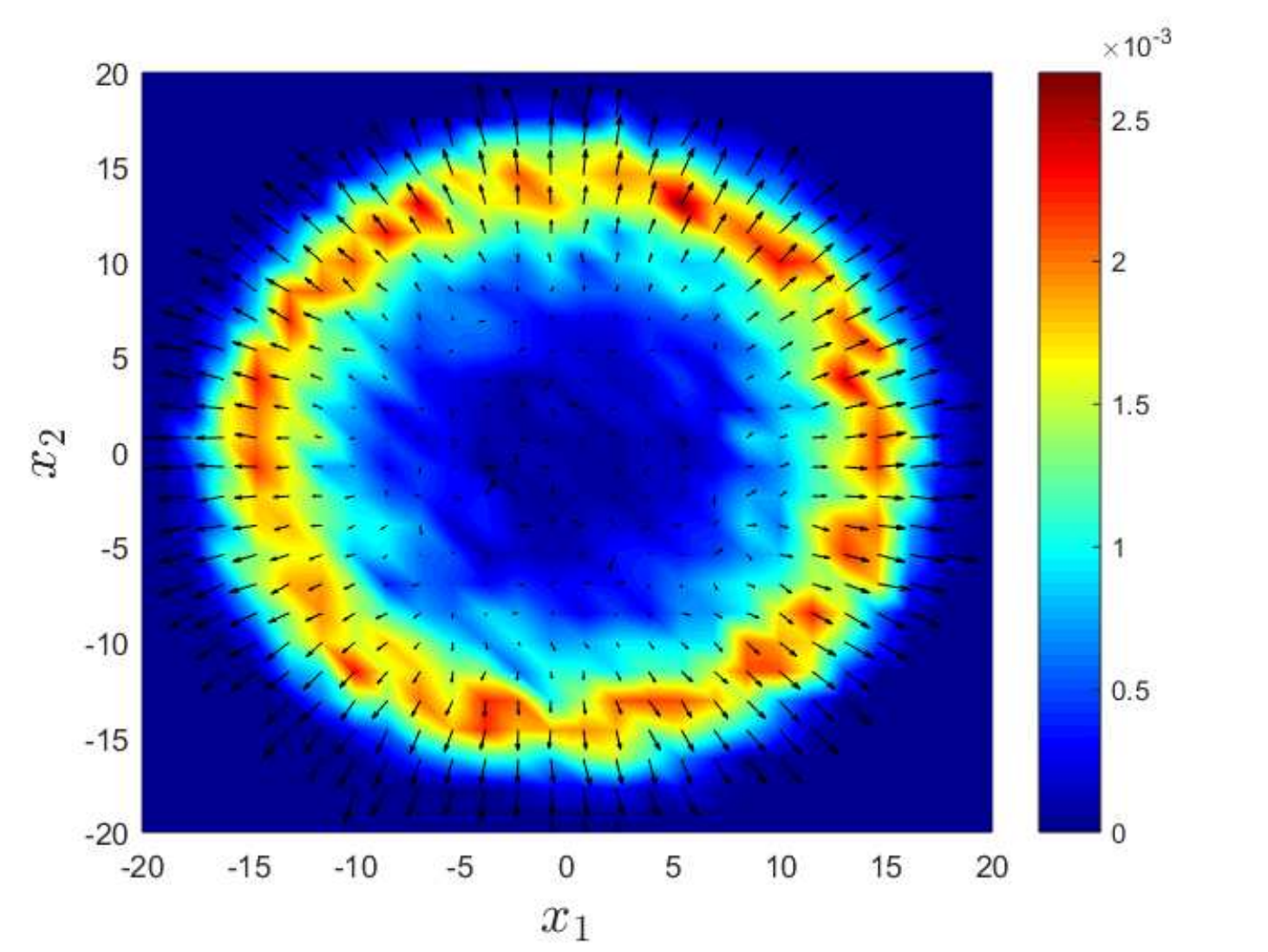}}
\subfigure[$t=20$]{\includegraphics[scale=0.420]{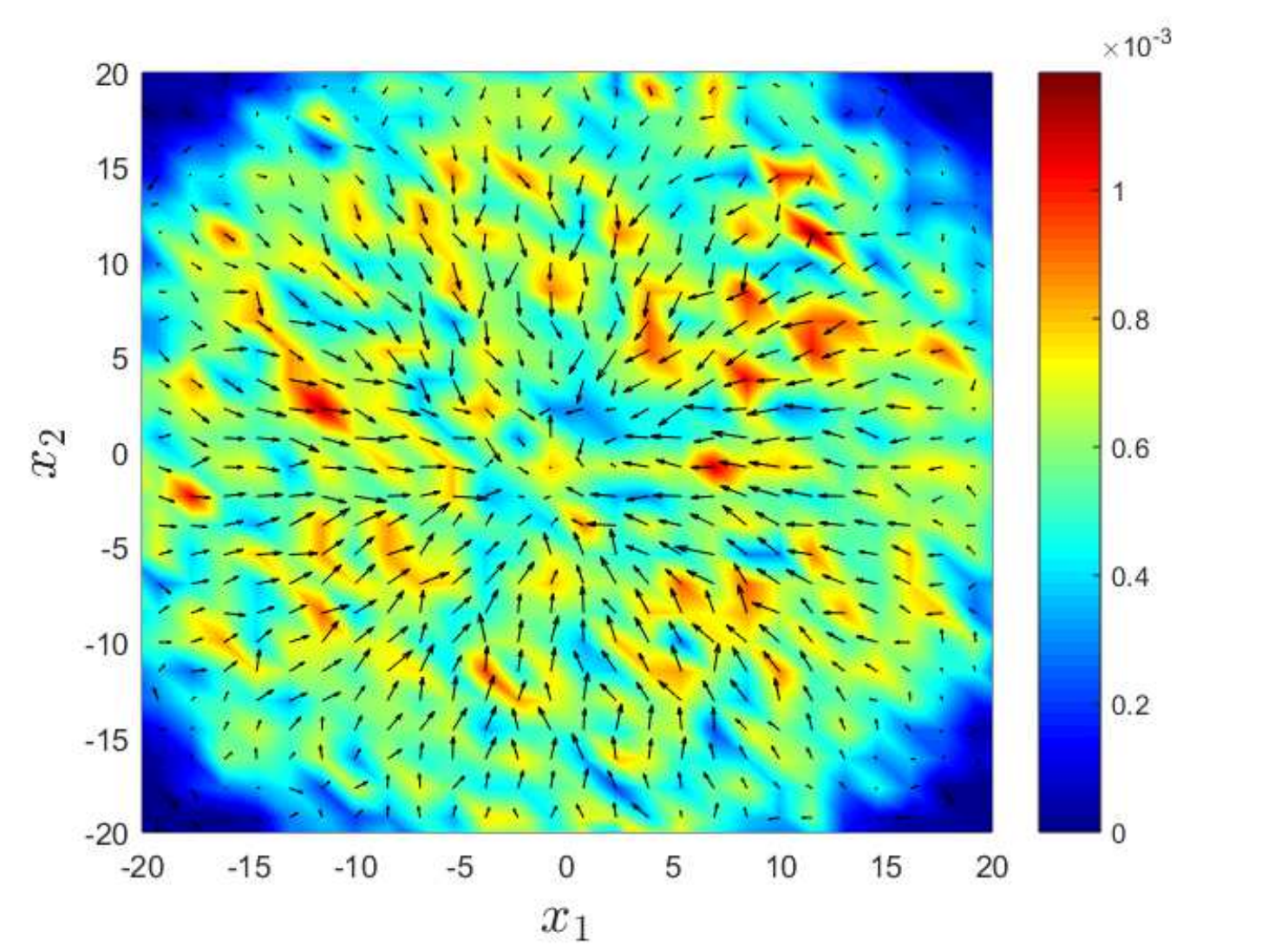}}\\
\subfigure[$t=150$]{\includegraphics[scale=0.420]{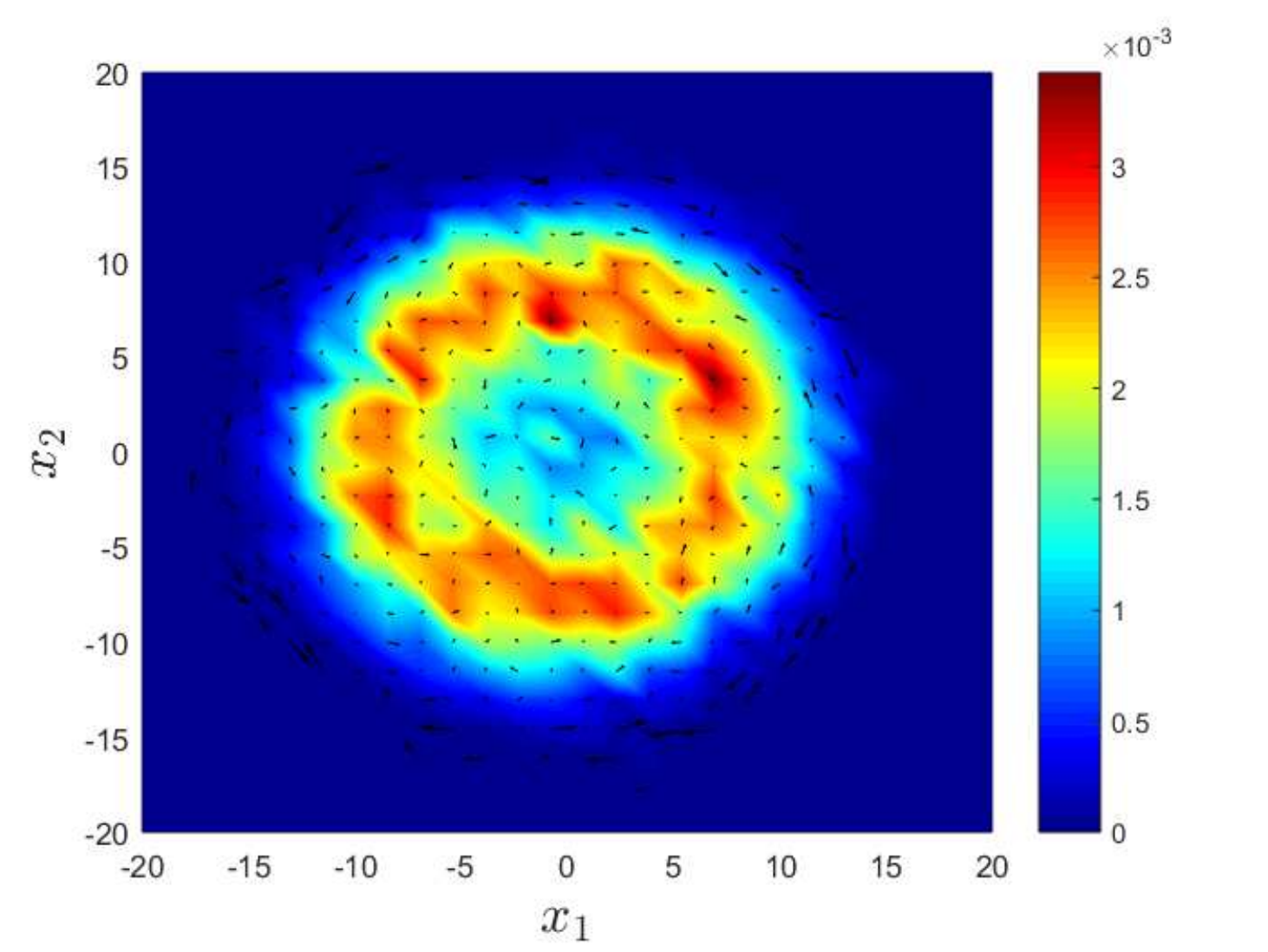}}
\subfigure[$t=200$]{\includegraphics[scale=0.420]{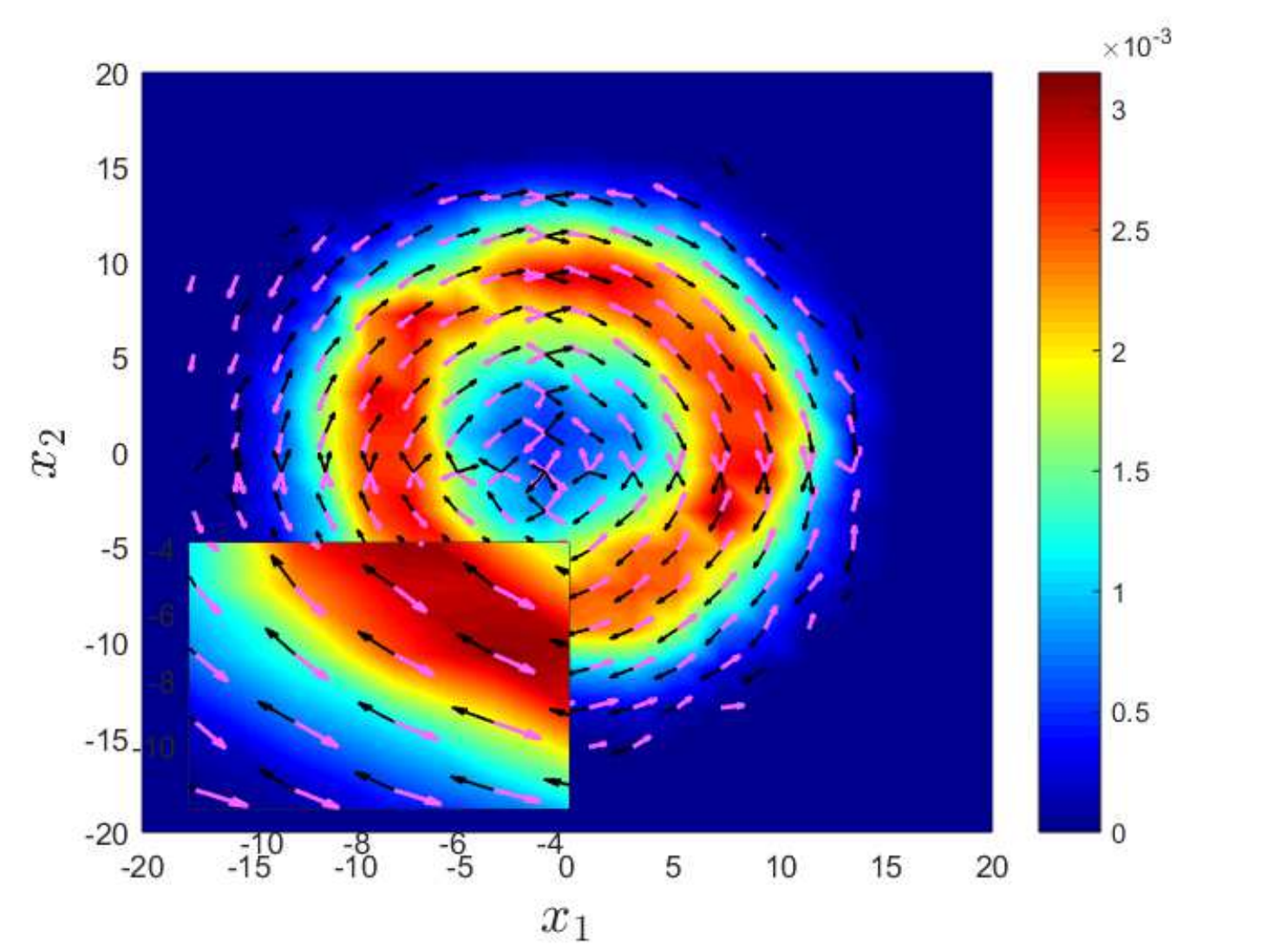}}
\caption{Evolution of the expected density from the D'Orsogna-Bertozzi et al. model with stochastic interactions through the MCgPC scheme over the time interval $t\in[0,200]$, $\Delta t=10^{-2}$. We considered $N=10^5$ agents, $S=10$ and a stochastic Galerkin decomposition of order $M=8$.  At $t=200$ the velocity field is reconstructed discriminating the orientation of the particles to highlight the emerging double mill structure. }
\label{fig:mill}
\end{figure}

\paragraph{D'Orsogna-Bertozzi et al. model}
In this paragraph we consider the D'Orsogna-Bertozzi et al. model to reproduce in the stochastic setting the typical mill dynamics described in \cite{CFTV,DOCBC}. According to what we introduced in Section \ref{sec:other_mod} we consider long-range attraction and short-range repulsion given by a stochastic Morse potential with $C_A=C_A(\theta)$, $C_R=C_R(\theta)$.

In Figure \ref{fig:mill} we present the evolution of the solution over $t\in[0,200]$ with the same initial data as in the previous example. In order to perform the MCgPC scheme we consider $N=10^5$ agents, the mean-field Monte Carlo is considered with $S=10$ interacting agents at each time step and the stochastic Galerkin projection uses $M=8$ terms. The stochastic Morse potential is given by $C_A(\theta) = 30+\theta$, $C_R(\theta) = 10+\theta$, with $\theta\sim \mathcal U([-1,1])$. Other typical parameters are the following $\ell_A=100$, $\ell_R=3$. 

From the reconstruction of the expected density, over a $50\times 50$ grid, we can observe the emergence of a double mill structure. This fact is reminiscent of the phenomena reported in \cite{CKR} in which single mills prefer to bifurcate to double mills for small noise added to the particle model. This interesting phase transition seems to be a robust behavior for general noisy data based on our present considerations of random parameters.

\subsection{2D Test with 2D uncertainty}

\begin{figure}
\centering
\subfigure[$t=0.5$]{\includegraphics[scale=0.45]{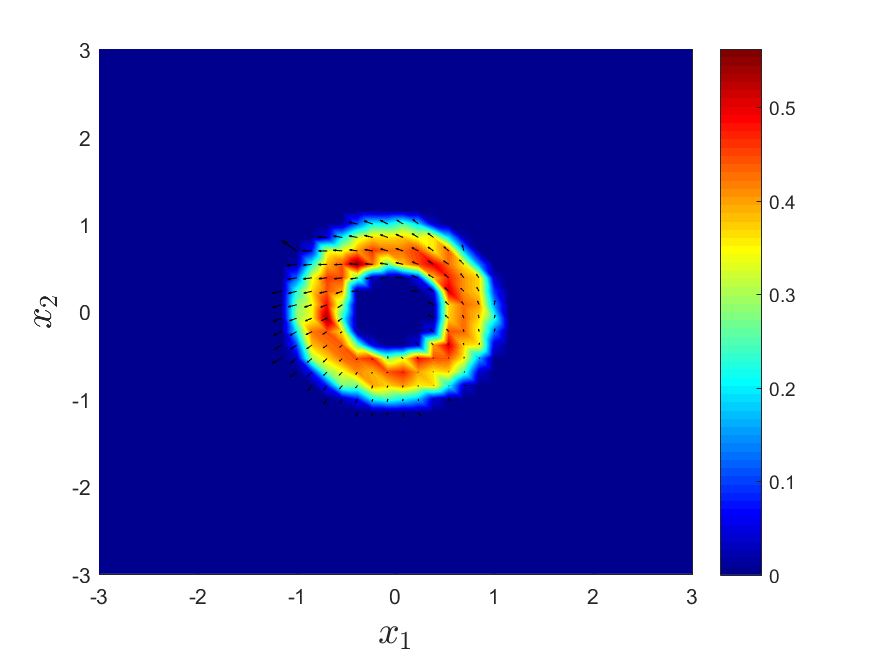}}
\subfigure[$t=1.5$]{\includegraphics[scale=0.45]{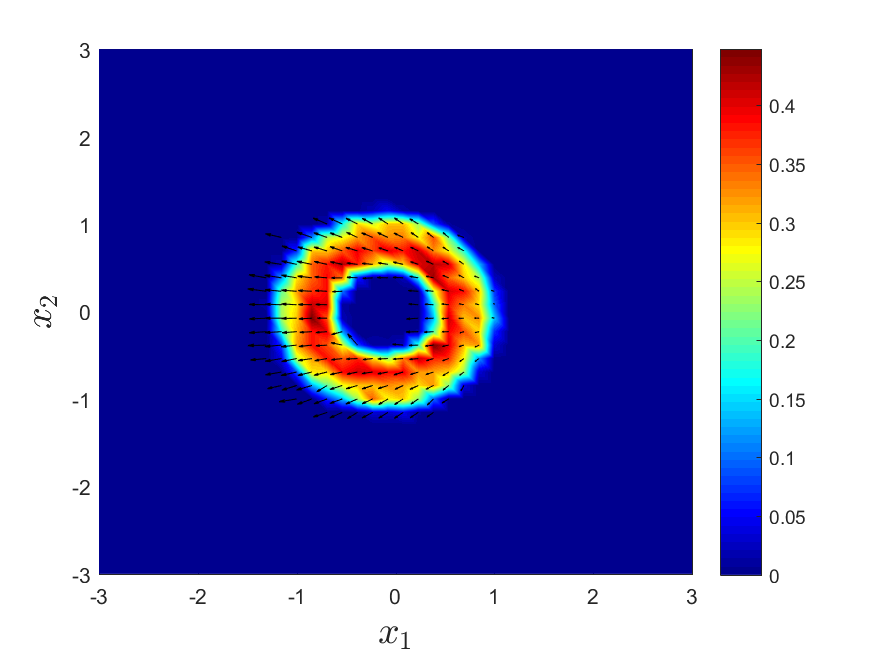}}\\
\subfigure[$t=2.5$]{\includegraphics[scale=0.45]{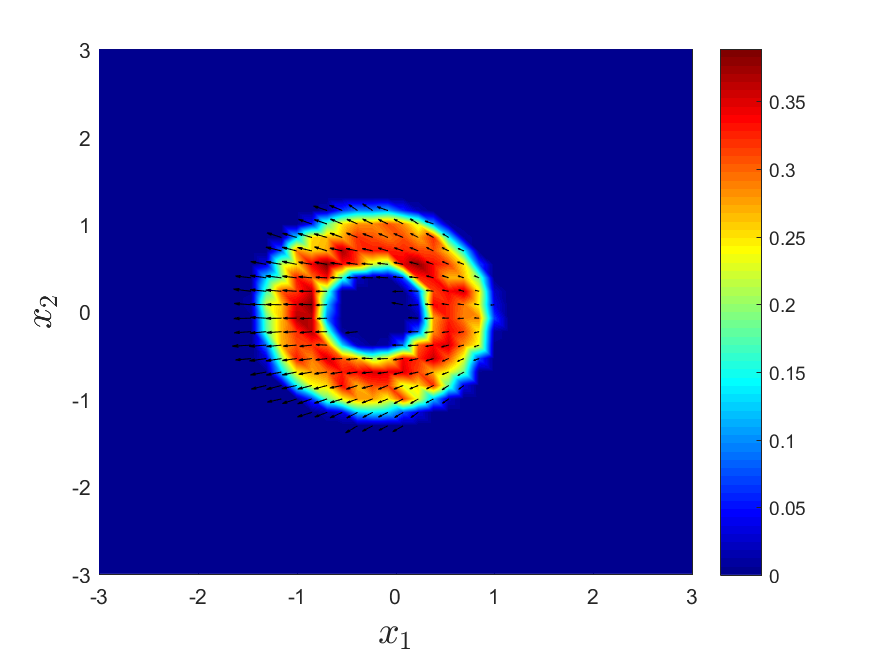}}
\subfigure[$t=5$]{\includegraphics[scale=0.45]{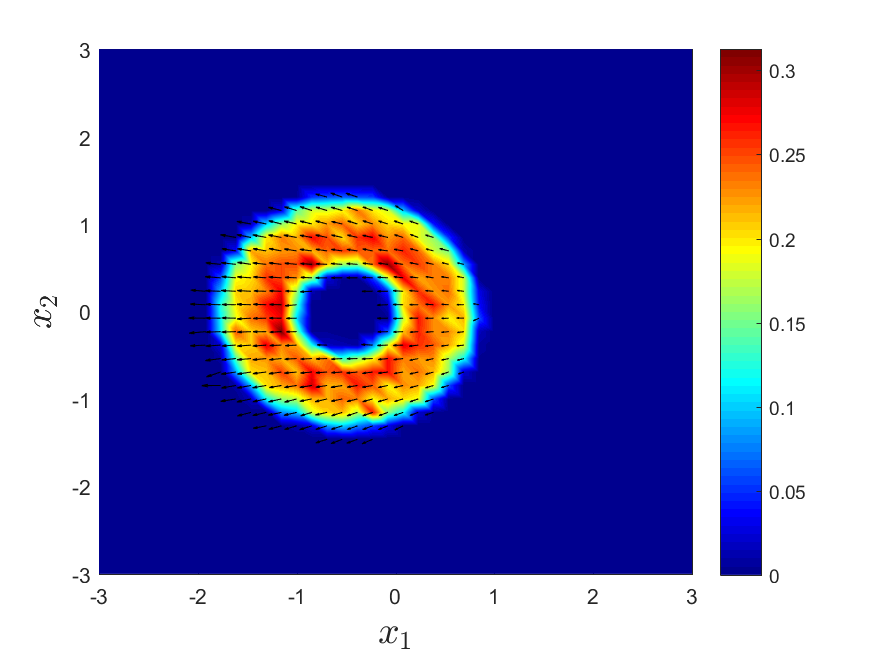}}
\caption{Evolution up to $t=5$ with $\Delta t = 10^{-2}$ of the expected density from the D'Orsogna-Bertozzi et al. model with alignment. We consider bivariate stochastic attraction repulsion strength, $C_A(\theta_2)=30+\theta_2$, $C_R(\theta_2)=10+\theta_2$, $\theta_2\sim \mathcal U([-1,1])$, and stochastic alignment dynamics $\gamma(\theta_1)=0.1+0.05\theta_1$ with  $(\theta_1,\theta_2)\sim g_1(\theta_1)g_2(\theta_2)$. Here we take $N=10^5$ agents, $S=5$ and $M=5$. }
\label{fig:CSDB}
\end{figure}

Within this section we investigate the case of the full particle system \eqref{eq:system}, where the dynamics of self-propulsion, attraction and repulsion are given by the D'Orsogna-Bertozzi et al. model whereas the alignment dynamics is given by the Cucker-Smale model. In particular, we concentrate on the case where the evolution of the system is affected by an uncorrelated 2D random term $\theta=(\theta_1,\theta_2)$, i.e. $\theta\sim g(\theta_1,\theta_2)=g_1(\theta_1)g_2(\theta_2)$. At the kinetic level the model is described by the evolution of the density function $f=f(\theta_1,\theta_2,x,v,t)$ solution of the following kinetic equation
\be\label{eq:CSDB}
\partial_t f + v\cdot \nabla_x f = \left[ (\nabla_x U* \rho )\cdot \nabla_v f + \nabla_v \cdot [\mathcal H[f]-((\alpha-\beta|v|^2)v)f] \right],
\ee
where the alignment term is given by the nonlocal operator
\[
\mathcal H[f](\theta_1,\theta_2,x,v,t) = \int_{\RR^2}\int_{\RR^2} \dfrac{K}{(1+|x-y|^2)^{\gamma(\theta_1)}}(v-w)f(\theta_1,\theta_2,x,v,t)dv\,dx,
\]
and the potential $U(\theta_2;|x-y|)$ depends only on $\theta_2$ as defined in \eqref{eq:U_Morse}. In order to apply the MCgPC method, we consider the 2D stochastic Galerkin decomposition   
\[
x^M (\theta_1,\theta_2,t) = \sum_{k,h=0}^M \hat x_{i,kh}\Phi_k(\theta_1)\Psi_h(\theta_2),\qquad v^M (\theta_1,\theta_2,t) = \sum_{k,h=0}^M \hat v_{i,kh}\Phi_k(\theta_1)\Psi_h(\theta_2),
\]
being $\{\Phi_k(\theta_1)\}_{k=0}^M$ and $\{\Psi_h(\theta_2)\}_{h=0}^M$ the orthogonal basis of the introduced random terms. The projection of the particle system is given for all $\ell,r=0,\dots,M$ by 
\[
\begin{cases}\vspace{1em}
\dfrac{d}{dt}{\hat x}_{i,\ell r} = \hat v_{i,\ell r},\\
\dfrac{d}{dt}{\hat v}_{i,\ell r} = \displaystyle\sum_{k,h=0}^M S^i_{\ell r k h} \hat{v}_{i,kh} + \dfrac{1}{N} \sum_{j=1}^N \sum_{k,h=0}^M E^{ij}_{\ell r k h} (\hat v_{j,kh}-\hat v_{i,kh})-\dfrac{1}{N} \sum_{j\ne i} B^{ij}_{\ell r},
\end{cases}
\]
being $S^i_{\ell r k h}$, $E_{\ell r k h}^{ij}$, $B_{\ell r }^{ij}$ defined as follows
\[
\begin{split}
S^i_{\ell r k h} &= \iint_{I_{\Theta_1}\times I_{\Theta_2}}(a-b|v_i^M|^2)\Phi_\ell(\theta_1)\Psi_r(\theta_2) \Phi_k(\theta_1)\Psi_h(\theta_2) dg_1(\theta_1)dg_2(\theta_2),\\
E^{ij}_{\ell r k h} &= \iint_{I_{\Theta_1}\times I_{\Theta_2}} H(\theta_1;|x_i^M-x_j^M|) \Phi_\ell(\theta_1)\Psi_r(\theta_2) \Phi_k(\theta_1)\Psi_h(\theta_2) dg_1(\theta_1)dg_2(\theta_2)\\
B^{ij}_{\ell r} &= \iint_{I_{\Theta_1}\times I_{\Theta_2}} \nabla_x U(\theta_2;|x_i^M-x_j^M|) \Phi_\ell(\theta_1)\Psi_r(\theta_2) dg_1(\theta_1)dg_2(\theta_2),
\end{split}
\]
with $U(\cdot)$ the Morse potential defined in \eqref{eq:U_Morse}. In Figure \ref{fig:CSDB}, we present the MCgPC solution in case of stochastic attraction repulsion strengths: $C_A(\theta_2)=30+\theta_2$, $C_R(\theta_2)=10+\theta_2$, $\theta_2\sim \mathcal U([-1,1])$, and alignment dynamics \eqref{eq:HCS} with parameters: $K=5.0$, $\gamma(\theta_1)=0.1+0.05\theta_1$,  $\theta_1\sim \mathcal U([-1,1])$. The initial data is the same as in the previous two examples. The computational cost of the method now is $O(NSM^2)$, therefore we considered $S=5$ and $M=4$ to obtain a similar cost of the case presented in Figure \ref{fig:mill}. 

 \section*{Conclusion}
In this paper we develop a novel approach for the construction of nonnegative gPC approximations of mean-field equations. The method is based on a Monte Carlo approximation of the kinetic mean-field equation in phase space combined with a gPC approximation of the random inputs on the particle samples. The idea presented here in principle admits several generalizations to other kinetic equations like Vlasov-Fokker-Planck equations and Boltzmann equations. 
These aspects will be the subject of future investigations and will be presented elsewhere.     
 
\section*{Acknowledgments}
JAC was supported by the Engineering and Physical Sciences Research Council (EPSRC) under grant no. EP/P031587/1. MZ acknowledges Imperial College (London, UK) and Institut Mittag-Leffler (Stockholm, Sweden) for the kind hospitality and the support from ``Compagnia di San Paolo'', Torino, Italy. LP and MZ acknowledge partial support by the research grant \emph{Numerical methods for uncertainty quantification in hyperbolic and kinetic equations} of the group GNCS of INdAM, Italy. 


\end{document}